\documentclass[a4paper,11pt,oneside,notitlepage]{article}
\usepackage{times}
\usepackage{lmodern}
\usepackage[T1]{fontenc}
\usepackage[utf8]{inputenc}
\usepackage[english]{babel}
\usepackage{amssymb}
\usepackage{amsmath}
\usepackage{mathtools} 
\usepackage{amsthm}
\numberwithin{equation}{section}
\usepackage{amsfonts, mathrsfs}
\usepackage{latexsym}
\usepackage[pdftex]{color, graphicx}
\usepackage{makeidx}
\usepackage{sidecap}
\usepackage{verbatim}
 \usepackage{hyperref}
 \usepackage{orcidlink}
\usepackage{geometry}
\usepackage{subfig}
\usepackage{tikz,float}
\usetikzlibrary{shapes,arrows,shadows}
\usepackage[displaymath,mathlines, running]{lineno} 
\usepackage{comment} 
\usepackage{esint} 
\usepackage{pdfsync}
\definecolor{vg}{rgb}{0.0, 0.40, 0.15}
\newtheorem{thm}{Theorem}[section]
\newtheorem{prop}[thm]{Proposition}
\newtheorem{lemma}[thm]{Lemma}

\newtheorem{remark}[thm]{Remark}

\newcommand{\media}[1]{- \hskip-.9em \int_{#1}}

\def\R{\mathbb{R}}

\def\N{\mathbb{N}}
\def\e{\varepsilon}

\def\hom{{\rm hom}}

\def\media
   {\mkern12mu\hbox{\vrule height4pt
    depth-3.2pt width5pt}
    \mkern-16.5mu\int}

 \newcommand{\mres}{\mathbin{\vrule height 1.6ex depth 0pt width
0.13ex\vrule height 0.13ex depth 0pt width 1.3ex}}

\title{Homogenization and 3D-2D dimension reduction of a functional on manifold valued BV space}
\author{Luca Lussardi \orcidlink{0000-0001-9130-3573}, Andrea Torricelli \orcidlink{0000-0002-8612-4076} and Elvira Zappale \orcidlink{0000-0001-7419-300X}}
\date{}

\AtEndDocument{\bigskip{\footnotesize%
  \textsc{L. Lussardi: Dipartimento di Scienze Matematiche G.L.\,Lagrange, Politecnico di Torino, C.so Duca degli Abruzzi 24, 10129 Torino, Italy} \par
  \textit{E-mail address}: \texttt{luca.lussardi@polito.it}
  \par
  \addvspace{\medskipamount}
  \textsc{A. Torricelli: Dipartimento di Scienze Matematiche G.L.\,Lagrange, Politecnico di Torino, C.so Duca degli Abruzzi 24, 10129 Torino, Italy} \par
  \textit{E-mail address}: \texttt{andrea.torricelli@polito.it}
  \par
  \addvspace{\medskipamount}
  \textsc{E. Zappale: Dipartimento di Scienze di base e Applicate per l'Ingegneria, Sapienza - Universit\`{a} di Roma,
via Antonio Scarpa, 16, 00161 Roma, Italy and
CIMA, Universidade de \'Evora, Portugal} \par
  \textit{E-mail address}: \texttt{elvira.zappale@uniroma1.it}
}}

\begin{document}
\maketitle


{\bf Abstract.}
We study the simultaneous homogenization and dimension reduction of an energy functional with linear growth defined on the space of manifold valued Sobolev functions. The study is carried out by $\Gamma$-convergence, providing an integral representation result in the space of manifold constrained functions with bounded variation.

\section{Introduction and main results}

The study of variational problems constrained to manifold-valued functions has attracted increasing attention due to its relevance in various fields such as micromagnetics and liquid crystals (see for instance \cite{GJ,P,SP80, Vir}). In this paper, we investigate the simultaneous process of homogenization and dimension reduction for a class of energy functionals with linear growth, defined on the space of functions taking values in a smooth manifold without boundary. Specifically, we consider a one-parameter family of functionals defined on $W^{1,1}$ functions with values in a compact, connected manifold $\mathcal{M} \subset \mathbb{R}^3$, and study their behavior under the joint asymptotic regime of periodic homogenization and thin domain limit.

The joining of this two processes is a classical field of study, indeed the simultaneous homogenization and dimension reduction of an integral functional defined on real valued Sobolev functions has been studied in \cite{BFF} in the case $p>1$, while the case $p=1$ can be covered by the Global Method \cite{BFM}, but without explicit formulae which connect the limiting energy density with the original ones. In the application framework, regarding either the solely study of homogenization or the study of thin structures we refer to \cite{A88, AMMZ22, AT23, BZZ08, BMbv, BM, BDFV, BF01, CMMO18, CMZ25, DFK21, DG95, DG05, FF12, FFV21, MMS15, MS14}
among a wider literature, Having in mind applications in the superlinear setting we refer, for instance, to \cite{CGO24, FG23, FIOM14, KK16, LTZ1}. As far as we know, there are no rigorous derivation on dimensional reduction for functionals with linear growth defined on manifold-valued Sobolev functions.

Our analysis is carried out in the framework of $\Gamma$-convergence (see \cite{DMbook}), which provides a rigorous tool for deriving effective models and capturing the asymptotic behavior of  infimizers. A key feature of our approach is the use of blow-up techniques and relaxation arguments adapted to the manifold setting.

The main result of the paper establishes the $\Gamma$-limit of the rescaled energy functional in the space $BV(\omega; \mathcal{M})$, where $\omega \subset \mathbb{R}^2$ represents the cross-section of the thin domain.

We assume that the domain is an inhomogeneous cylinder, whose microstructure is distributed with periodicity within the material described by the small parameter $h$ comparable with the height of the domain. The equilibrium configurations are detected as minimizers of an
integral functional of the form 
\[
\int_{\omega\times (-\frac{h}{2},\frac{h}{2})} f\left ( \frac{x}{h}, \nabla u\right ) \, dx \qquad u: \omega\times \left (-\frac{h}{2},\frac{h}{2} \right ) \rightarrow \mathbb{R}^3,
\]
under suitable boundary conditions, where $\omega \subset \mathbb{R}^2$ is a bounded open set, and $f: \mathbb{R}^3 \times \mathbb{R}^{3 \times 3}\rightarrow \mathbb{R}$
is a periodic integrand with respect to the first variable, and $u$ is a manifold-valued Sobolev field, that will be specialized in the sequel. 
Due to the many applications, it is worth to recall that solely the homogenization of integral functionals depending on $x$ and $\nabla u$ and defined on manifold-valued Sobolev fields has been studied in \cite{BM} for $u\in W^{1,p}$ and in \cite{AEL07, BMbv} for $u\in W^{1,1}$, we also refer to \cite{P} for other related models.
Analogously the dimensional reduction of micromagnetic and ferromagnetic energy has been studied, in several contexts, we recall \cite{AL01, BZ07, CCG, DKMO, GH, GH2} among the others.\\
\color{black}

\noindent
Given a Carathéodory function $f:\R^3\times\R^{3\times 3}\to \R$ we consider the functional
\[
\frac{1}{h}\int_{\omega_{,h}} f \left (\frac{x}{h}, \nabla u\right ) \, dx, \qquad u \in W^{1,1}(\omega_{,h}; \mathcal{M}),
\]
with $\omega_{,h} := \omega \times \left (- \frac{h}{2}, \frac{h}{2} \right ),$ $h>0,$  $\omega \subset \mathbb{R}^2$ open and bounded, and $\mathcal{M}$ a smooth submanifold of $\mathbb{R}^3$ ithout boundary. In particular, here and in the sequel, when not stated otherwise, we always assume that $f$ has the following properties:
\begin{itemize}
    \item [(H1)] $f(\cdot, x_3, \xi)$ is 1-periodic, i.e. for every $(x_{\alpha}, x_3) \in \mathbb{R}^3$ and $\xi \in \mathbb{R}^{3 \times 3}$ it holds
    \begin{equation*}
    \label{H1}
        f(x_{\alpha} + \mathbf{e}_i, x_3, \xi) = f(x_{\alpha}, x_3, \xi), \qquad \forall i = 1,2
    \end{equation*}
    where $\{\mathbf{e}_1, \mathbf{e}_2\}$ is the canonical basis of $\mathbb{R}^2;$

    \item [(H2)] $f$ has linear growth: there exists $\alpha, \beta > 0$ such that
    \begin{equation}
    \label{H2}
        \alpha |\xi| \le \, f(x, \xi) \le \, \beta (1 + |\xi|),
    \end{equation}
    for a.e. $x\in\R^3$ and for every $\xi\in\R^{3\times 3}$;
    \item [(H3)]  $f$ is a Lipschitz continuous function with respect to the second variable uniformly in the first variable, i.e. there exists $L>0$ such that it holds
    \begin{equation*}
    \label{H3}
        |f(x, \xi) - f(x, \eta)| \le \, L \, |\xi - \eta|
    \end{equation*}
    for a.e. $x\in\mathbb{R}^3$ and $\xi,\eta\in\R^{3\times 3}$;\item[(H4)] there exist $C>0$ and $0<q<1$ such that
    \begin{equation*}
        |f(x,\xi)-f^\infty(x,\xi)|\le C(1+|\xi|^{1-q}),
    \end{equation*}
    for a.e. $x\in\R^3$ and $\xi\in\R^{3\times 3}$ and where $f^\infty:\R^3\times\R^{3\times 3}\to \R$ is the (strong) recession function of $f$ and it is defined as
    \begin{equation}
    \label{recessionfunction}
    f^\infty(x,\xi):=\limsup_{t\to+\infty}\frac{f(x,t\xi)}{t}.
    \end{equation}
    We observe that by \eqref{H2} it holds
\begin{equation}
\label{recessiongrowth}
    \alpha |\xi| \le \, f^\infty(x, \xi) \le \, \beta |\xi|,
\end{equation}
for a.e. $x\in \R^3$ and every $\xi \in \R^{3\times 3}$. 
\end{itemize}
We consider the functional $\tilde{I}^h: L^p(\omega_{,h}; \mathcal{M}) \rightarrow \overline{\mathbb{R}}$
\[
\tilde{I}^h(u) := \left \{
\begin{array}{lll}
\!\!\!\!\!\! & \displaystyle \frac{1}{h}\int_{\omega_{,h}} f \left (\frac{x}{h}, \nabla u\right ) \, dx \qquad & \textnormal{if $u \in W^{1,1}(\omega_{,h}; \mathcal{M})$}\\[4mm]
\!\!\!\!\!\! & + \infty &\textnormal{elsewhere.}
\end{array}
\right.
\]
By a change of variable, to study the $\Gamma-$limit of $\tilde{I}^h$ is equivalent to study the $\Gamma-$limit of the rescaled functional $I^h$ defined as
\begin{equation}
\label{functional}
    I^h(u) := \left \{
\begin{array}{lll}
\!\!\!\!\!\! & \displaystyle \int_{\Omega} f \left (\frac{x_{\alpha}}{h}, x_3, \nabla_h u\right ) \, dx \qquad & \textnormal{if $u \in W^{1,1}(\Omega; \mathcal{M})$}\\[4mm]
\!\!\!\!\!\! & + \infty &\textnormal{elsewhere}
\end{array}
\right.
\end{equation}
with $\Omega := \omega \times \left (- \frac{1}{2}, \frac{1}{2} \right ) = \omega_{,1}$ and $\nabla_h:=[\frac{\partial}{\partial x_1}, \frac{\partial}{\partial x_2}, \frac{1}{h}\frac{\partial}{\partial x_3}]$.
From here onward, we denote $\nabla_\alpha:=[\frac{\partial}{\partial x_1}| \frac{\partial}{\partial x_2}]$ and $\nabla_3:=\frac{\partial}{\partial x_3}$, so that $\nabla_h=[\nabla_\alpha| \frac{1}{h}\nabla_3]$. Moreover, we denote with $\xi_\alpha$ an element of $\R^{3\times 2}$ and with $\xi$ an element of $\R^{3\times 3}.$

\begin{thm}
\label{BVcase}
    Assume $\mathcal{M}$ compact and connected, and let  $f:\R^3\times\R^{3\times 3}\to \R$ be a Carath\'eodory function satisfying {\rm (H1)}, {\rm (H2)}, {\rm (H3)}, and {\rm (H4)}. Then, the $\Gamma-$limit of  $I^h$ 
    with respect to the $L^1$-topology is the functional $I$ defined as 
    \begin{equation}
        \label{BVcandidate}
        I(u):=\int_\omega Tf^0_{\rm hom}(u,\nabla_\alpha u)dx_\alpha +\int_\omega Tf^{0,\infty}_{\rm hom}\left({u}, \frac{d D^c_\alpha {u}}{d|D^c_\alpha{u}|}\right)d|D^c_\alpha{u}|+\int_{\omega\cap S_u} \theta(u^+, u^-,\nu_u)d\mathcal{H}^{1},
    \end{equation}
    if $u\in BV(\omega;\mathcal{M})$ and $I(u)=+\infty$ elsewhere, with $T f^0_{\rm hom}: \mathbb{R}^3 \times \mathbb{R}^{3 \times 2} \rightarrow \mathbb{R}$ defined as
\begin{align}
\label{Tfhom}
&Tf^0_{\rm hom}(s,\xi_\alpha) :=  \liminf_{t \rightarrow + \infty} \inf_{\varphi} \Bigg \{\frac{1}{t^2} \int_{(tQ')_{,1}} f(x_{\alpha}, x_3, \xi_\alpha + \nabla_{\alpha} \varphi | \nabla_3 \varphi) \, d x_{\alpha} d x_3: \nonumber\\
& \, \varphi \in W^{1,\infty}((tQ')_{,1}; T_s (\mathcal{M})), \,\, \varphi(x_{\alpha}, x_3) = 0\,\,\, \textnormal{for every $(x_{\alpha}, x_3) \in \partial (tQ') \times \left(-\frac{1}{2},\frac{1}{2}\right)$} \Bigg \},
\end{align}
where $Q':=\left(-\frac{1}{2},\frac{1}{2}\right)^2,$ $(tQ')_{,1}:=tQ'\times \left(-\frac{1}{2},\frac{1}{2}\right), $ and $T_s (\mathcal{M})$ denotes the tangent space to $\mathcal{M}$ in $s$. Moreover, for every $s \in \mathbb R^3$, $Tf^{0,\infty}_{\rm hom}(s,\cdot)$ is the (strong) recession function of $Tf^0_{\rm hom}(s,\cdot)$,
    i.e. for every $(s,\xi_\alpha)\in\mathbb R^3 \times \mathbb R^{3\times 2}$:
    \begin{equation}\label{ohominfty}
    Tf^{0,\infty}_{\rm hom}(s,\xi_\alpha)= (Tf^0_{\rm hom})^\infty(s,\xi_\alpha)= \limsup_{t \to +\infty} \frac{Tf^0_{\rm hom}(s,t\xi_\alpha)}{t},
    \end{equation}
    while $\theta:\mathcal{M}\times\mathcal{M}\times\mathbb{S}^1\to \mathbb R$, is defined as
    \begin{eqnarray}
        \label{jumpterm}
        & \theta(a,b,\nu):= \lim_{t\to+\infty} \inf_{\phi}\Bigg \{\displaystyle  \frac{1}{t}\int_{(tQ_{\nu})_{,1}} f^\infty(x_\alpha,x_3,\nabla \phi)dx_\alpha dx_3: \nonumber \\
        &\,\,\, \phi \in\ W^{1,1}((tQ_{\nu})_{,1}, \mathcal{M}), \phi=a\text{ on } \partial (tQ_\nu) \times \left(-\frac{1}{2},\frac{1}{2}\right)\cap \{x_\alpha\cdot\nu>0\} \text{ and }\\
        &\,\,\,\phi=b  \text{ on }  \partial (tQ_\nu)\times \left(-\frac{1}{2},\frac{1}{2}\right) \cap \{x_\alpha\cdot\nu<0\} \Bigg\}, \nonumber
    \end{eqnarray}
     where $Q_\nu$ is an open cube of $\R^2$ centered at the origin and with two of its faces parallel to $\nu$, $Q_{\nu,1}:=Q_\nu\times\left(-\frac{1}{2},\frac{1}{2}\right)$, and $\mathbb{S}^1$ is the unit sphere in $\mathbb{R}^2.$
\end{thm}

\section{Notation and preliminaries}\label{NP}
This section is devoted to fix notation, recall previous results that will be useful in the sequel and establish properties of the energy densities appearing in the main results. In what follows $\Omega:=\omega\times (-\frac{1}{2}, \frac{1}{2})$, $\omega \subset \mathbb{R}^2$ is open and bounded, and $\mathcal{M}$ is a smooth submanifold of $\R^3$ without boundary, further assumptions on $\mathcal{M}$ will be stated explicitly if needed. Given $s\in\mathcal{M}$, we write $T_s(\mathcal{M})$ for the tangent space to $\mathcal{M}$ in $s$. Given $a, b \in \mathcal{M},$ we introduce the family of geodesic curves between $a$ and $b$ by
\begin{align}
\label{geodesic}
& \mathcal{G}(a,b) :=  \Big\{ \gamma \in \mathcal{C}^{\infty}(\mathbb{R}; \mathcal{M}): \gamma(t) = a, \,\,\, {\rm if} \,\,\,t \ge 1/2, \nonumber\\
&  \gamma(t) = b, \,\,\, {\rm if} \,\,\,t \le -1/2, \,\,\, \int_{\mathbb{R}} |\dot{\gamma}| \, dt = {\bf d}_{\mathcal{M}}(a,b) \Big \},
\end{align}
where ${\bf d}_{\mathcal{M}}$ denotes the geodesic distance on $\mathcal{M}.$
\\
We denote by $\mathcal{A}(\Omega)$ the family of all open subsets of $\Omega$ and with $\mathcal{A}(\omega)$ the family of all open subset of $\omega$. We write $B^k(s,r)$ for the closed ball in $\R^k$, $k\in\N$, of center $s\in\R^k$ and radius $r>0$. Moreover, we denote by $Q$ the cube $(-\frac{1}{2}, \frac{1}{2})^3$ and with $Q(x_0,\rho)$ the rescaled and translated cube $x_0+\rho Q$, with $x_0\in\R^3$, $\rho>0$. In a similar way, given $\nu \in \mathbb{S}^1,$ $Q_{\nu}$ stands for the open unit square in $\mathbb{R}^2$ centered at the origin, with the two sides orthogonal to $\nu;$ we set $Q_{\nu}(x_0, \rho) := x_0 + \rho Q_{\nu}.$ 
\\
Furthermore, given a set $A\subset\omega$ we denote by $A_{,h}$, with $h>0$, the set $A\times (-\frac{h}{2}, \frac{h}{2})$ and so $A_{,1}:=A\times(-\frac{1}{2}, \frac{1}{2})$. It follows that, in particular, $\Omega=\omega_{,1}$. We also denote by $Q'$ the cube $(-\frac{1}{2}, \frac{1}{2})^2$, so $Q'_{,h}=(-\frac{1}{2}, \frac{1}{2})^2\times (-\frac{h}{2}, \frac{h}{2})$ for $h>0$, while $Q'_{,1}=(-\frac{1}{2}, \frac{1}{2})^3=Q$ and $Q_{\nu,1} = Q_{\nu} \times (-\frac{1}{2}, \frac{1}{2}).$ $\mathcal{M}(\Omega)$ is the space of real valued Radon measure in $\Omega$ with finite total variation,  $\mathcal{L}^k$, is the $k$-dimensional Lebesgue measure, with $k\in\N$. Finally, given $\lambda,\mu\in\mathcal{M}(\Omega)$ we denote by $\frac{d\lambda}{d\mu}$ the Radon-Nykod\'ym derivative of $\lambda$ with respect to $\mu$. By a generalization of Besicovitch Differentiation Theorem, see for instance \cite[Proposition 2.2]{Ambrosio-Dal Maso}, there exists a Borel set $E$ such that $\mu(E)=0$ and $\frac{d\lambda}{d\mu}(x)=0$ for every $x\in supp\, \lambda \setminus E.$\\

\noindent
Given $s\in\mathcal{M}$, we consider the orthogonal projection \begin{equation}\label{orthP}P_s: \R^3\to T_s(\mathcal{M}),
\end{equation} and we define the function 
$\mathbf{P}_s: \R^{3\times 3}\to [T_s(\mathcal{M})]^3$ as
\begin{equation}\label{3dorthP}
\mathbf P_s(\xi):=\left[P_s(\xi_1)| P_s(\xi_2)| P_s(\xi_3)\right],
\end{equation}
for every $\xi=[\xi_1|\xi_2|\xi_3]\in\R^{3\times 3}.$ For a Carathéodory function $f:\mathbb R^3 \times \mathbb R^{3 \times 3}\to \mathbb R$, we set
\begin{equation}
\label{perturbedf}
    \Bar{f}(x,s,\xi):=f(x,\mathbf{P}_s(\xi))+|\xi-\mathbf{P}_s(\xi)|^p.
\end{equation}
This function will play a crucial role in our subsequent analysis, since it will appear in the formulas to detect our limiting energy densities.
 By construction the function $\Bar{f}:\R^3\times\mathcal M\times\R^{3\times 3}\to\R$ is Carathéodory, i.e. it is measurable with respect to the first variable and continuous with respect to the last two variables. Moreover, if conditions (H1) and (H2) are satisfied then $\Bar{f}$ also satisfies (H1) and (H2), i.e. $\Bar{f}$ is $1$-periodic in the first variable and there exists $C>0$ such that
 \begin{equation}\label{pgrowth}
 \frac{1}{C}|\xi|\le\Bar{f}(x,s,\xi)\le C(1+|\xi|),
 \end{equation}
for every $(s,\xi)\in \mathcal M\times\R^{3\times 3}$ and for a.e. $x\in\R^3$. Now we extend the function $\bar{f}$ in \eqref{perturbedf} to the whole space $\mathbb{R}^3 \times \mathbb{R}^3 \times \mathbb{R}^{3\times 3}$. This can be done as in \cite[Lemma 3.1]{BMbv}. 
For $\delta_0 > 0$ fixed, let $\mathcal U :=\{
s \in \mathbb R^3 : {\rm dist}(s,\mathcal M) < \delta_0\}$ be the $\delta_0$-neighborhood of $\mathcal M$. Choosing $\delta_0 > 0$
small enough, we may assume that the nearest point projection $\Pi : \mathcal U\to \mathcal M$ is a well defined Lipschitz mapping.
Then the map $s \in \mathcal U \mapsto P_{\Pi(s)}$ is Lipschitz, with $P_{\Pi(s)}$ as in \eqref{orthP}. We denote by $\chi$ a cut-off function in $C^\infty_c (\mathbb R^3; [0, 1])$ such that
$\chi(t) = 1$ if ${\rm dist}(t,\mathcal M) \leq \delta_0/2$, and $\chi (t) = 0$ if ${\rm dist}(t,\mathcal M) \geq \frac{3 \delta_0}{4}$. We define
$\mathbb P_s(\xi) := \chi(s){\bf P}_{\Pi(s)}(\xi)$ for $(s, \xi) \in \mathbb R^3 \times \mathbb R^{3 \times 3},$ with ${\bf P}_{\Pi(s)}$ as in \eqref{3dorthP}.
In the sequel we will consider the integrand $g : \mathbb R^3 \times \mathbb R^3 \times \mathbb R^{3\times 3} \to [0, +\infty)$ given by
\begin{align}\label{g}
g(y, s, \xi) = f(y, \mathbb P_s(\xi)) + |\xi -\mathbb P_s(\xi)|.
\end{align}

\noindent
We also recall an important property of manifold-valued Sobolev functions. The following Theorem has been proved in \cite{Bet91, BZ88} and we will use it to prove Theorem \ref{BVcase}.
\begin{thm}
\label{label}
Let $\Omega\subset\R^3$ be open and bounded, let $\mathcal{S}$ be the family of all finite unions of $1$-dimensional manifolds in $\R^3$ and let $\pi_1(\mathcal{M})$ be the fundamental group of $\mathcal{M}$. Denote by $\mathcal{D}(\Omega, \mathcal{M})\subset W^{1,1}(\Omega, \mathcal{M})$ the set
    \begin{equation*}
        \mathcal{D}(\Omega, \mathcal{M}):=
        \begin{cases}
             W^{1,1}(\Omega, \mathcal{M})\cap C^\infty(\Omega, \mathcal{M}) \quad \text{if } \pi_1(\mathcal{M})=0\\
            \{u\in W^{1,1}(\Omega, \mathcal{M})\cap C^\infty(\Omega\setminus\Sigma, \mathcal{M}) \text{ for some } \Sigma\in \mathcal{S}\} \quad \text{otherwise}
        \end{cases}
    \end{equation*}
    Then $\mathcal{D}(\Omega, \mathcal{M})$ is dense in $ W^{1,1}(\Omega, \mathcal{M})$ with respect to the strong topology of $ W^{1,1}(\Omega, \R^3)$.
\end{thm}

\begin{remark}
\label{abuso}
    Given $A\subset\mathcal{A}(\omega)$, the set
\[
V(A) := \left \{v \in W^{1,p}(A_{,1}; \mathcal{M}): \frac{\partial v}{\partial x_3} = 0 \,\,\, \text{a.e. on } A_{,1} \right \},
\]
is isomorphic to the Sobolev space $W^{1,p}(A; \mathcal{M}).$  
\end{remark}

\subsection{Bulk part density}
Now we characterize the density of the bulk part of the $\Gamma$-limit. First, we recall that the extended function $g$ defined in \eqref{g} has the following properties.

\begin{lemma}{\cite[Lemma 3.1]{BMbv}}
\label{extension}
Assume that $\mathcal{M}$ is compact. Let $f: \mathbb{R}^3 \times \mathbb{R}^{3 \times 3} \rightarrow [0, + \infty)$ be a Carath\'eodory function satisfying {\rm (H1)}, {\rm (H2)}, and {\rm (H3)}. Then the function $g: \mathbb{R}^3 \times  \mathbb{R}^3 \times \mathbb{R}^{3 \times 3} \rightarrow [0, + \infty)$ in \eqref{g}, turns out to be a Carath\'eodory function such that
\begin{equation}
\label{(2.6)}
g(y, s, \xi) = f(y, \xi),\quad g^\infty(y,s,\xi)=f^\infty(y,\xi),
\end{equation}
for a.e. $y\in\R^3$, for every $s \in \mathcal{M}$ and $\xi \in [T_s(\mathcal{M})]^3$.
Furthermore it satisfies
\\
{\rm (i)} $g(y_\alpha,y_3, s, \xi)$ is 1-periodic in the first variable, with $(y_\alpha, y_3)\in \R^3$;
\\
{\rm (ii)} there exist $0 < \alpha' \le \beta'$ such that
\begin{equation}
\label{(2.7)}
\alpha' |\xi| \le \, g(y, s, \xi) \le \, \beta' (1 + |\xi|),
\end{equation}
for a.e. $y \in \mathbb{R}^3$ and for every $(s, \xi) \in \mathbb{R}^3 \times \mathbb{R}^{3 \times 3}$;
\\
{\rm (iii)} there exist constants $C > 0$ and $C' > 0$ such that
\begin{equation}
\label{(2.8)}
|g(y, s, \xi) - g(y, s', \xi)| \le \, C \, |s - s'| \, |\xi|,
\end{equation}
and
\begin{equation}
\label{roba}
    |g(y, s, \xi) - g(y, s, \xi')| \le \, C' \, |\xi - \xi'|,
\end{equation}
for a.e. $y \in \mathbb{R}^3$,  for every $s, s' \in \mathbb{R}^3,$ and for every $\xi \in \mathbb{R}^{3 \times 3}$;
\\
{\rm (iv)} if {\rm (H4)} holds, there exists $0<q<1$ and $C''$ such that
\begin{equation}
\label{BMbv(3.6)}
    |g(y,s, \xi)-g^\infty(y,s,\xi)|\le C''(1+|\xi|^{1-q}),
\end{equation}
for a.e. $y\in\R^3,$ for every $(s,\xi)\in\R^3\times\R^{3\times 3},$ and  where $g^\infty:\R^3\times\R^{3\times 3}\to \R$ is the (strong) recession function of $g$ and it is defined as
\[
g^{\infty}(y, s,\xi):=\limsup_{t\to+\infty}\frac{g(y,s, t\xi)}{t}.
  \]
\end{lemma}
Now we prove some properties of the bulk part $Tf_{hom}^0$ (see \eqref{Tfhom}) of $I$ (see \eqref{BVcandidate}).

\begin{prop}
    \label{characterization}
   Given a Carathéodory function $f:\R^3\times\R^{3\times 3}\to\R$ satisfying conditions {\rm (H1)} and {\rm (H2)}, let $g:\mathbb{R}^3 \times  \mathbb{R}^3 \times \mathbb{R}^{3 \times 3} \rightarrow [0, + \infty)$ be the functions defined as in \eqref{g}, then following properties hold.
    \begin{itemize}
        \item[{\rm (i)}] For every $s\in\mathcal{M}$ and $\xi_\alpha\in[T_s(\mathcal{M})]^2$
        \begin{equation}
        \label{(2.3)}
        Tf_{\rm hom}^0(s, \xi_\alpha)=g_{\rm hom}^0(s,\xi_\alpha),
        \end{equation}
        where $Tf^0_{\rm hom}$ is defined as in \eqref{Tfhom} and
        \begin{align}
            &g^0_{\rm hom}(s,\xi_\alpha) =  \lim_{t \rightarrow + \infty} \inf_{\varphi} \Bigg \{\frac{1}{t^2} \int_{(tQ')_{,1}} g(x_{\alpha}, x_3,s, \xi_\alpha + \nabla_{\alpha} \varphi | \nabla_3 \varphi) \, d x_{\alpha} d x_3: \label{f0hom}\\
& \, \varphi \in W^{1,\infty}((tQ')_{,1}; \R^3), \,\, \varphi(x_{\alpha}, x_3) = 0\,\,\, \textnormal{for every $(x_{\alpha}, x_3) \in \partial (Q') \times \left(-\frac{1}{2},\frac{1}{2}\right)$} \Bigg \} \nonumber
        \end{align}
        
        \item[{\rm (ii)}] The function $Tf_{\rm hom}^0$ is tangentially quasiconvex in the second variable, i.e.
        \[
        Tf_{\rm hom}^0(s,\xi_\alpha)\le\int_{Q'} Tf_{\rm hom}^0(s, \xi_\alpha+ \nabla_\alpha \psi)dx_\alpha,
        \]
        for every $s\in\mathcal{M}, \xi_\alpha\in [T_s(\mathcal{M})]^2$, and $\psi\in W^{1,\infty}_0(Q'; T_s(\mathcal{M}))$. In particular, $Tf_{\rm hom}^0(s, \cdot)$ is rank one convex.
        \item[{\rm (iii)}] $Tf_{\rm hom}^0$ is uniformly $1$-coercive and has uniform linear growth in the second variable (i.e. it satisfies inequalities as in \eqref{pgrowth}, uniformly with respect to $s$). Moreover, there exists $C>0$ such that for every $s\in\mathcal{M}$ and $\xi_\alpha, \xi_\alpha'\in [T_s(\mathcal{M})]^2$  
        \begin{equation*}
        |Tf_{\rm hom}^0(s, \xi_\alpha) - Tf_{\rm hom}^0(s, \xi_\alpha')|\le C|\xi_\alpha-\xi_\alpha'|.
        \end{equation*}

        \item[{\rm (iv)}] Finally, if hypothesis (H4) is also satisfied, then \begin{equation*}
            |Tf^0_{\rm hom}(s,\xi_\alpha)-Tf^{0,\infty}_{\rm hom}(s,\xi_\alpha)|\le C(1+|\xi_\alpha|^{1-q}),
        \end{equation*} 
        for some positive constant $C$ and for every $s\in\mathcal{M}$ and $\xi_\alpha \in [T_s(\mathcal{M})]^2$, where $Tf^{0,\infty}_{\rm hom}$ is the function appearing in \eqref{ohominfty}.
    \end{itemize}
\end{prop}

\begin{proof}
    The proof of (i), (ii), and (iii) follows from \cite[Proposition 2.1]{LTZ1}. Now we prove (iv) following the technique of \cite[Proposition 3.1-(v)]{BMbv}. Consider a sequence $(t_n)_n\subset\R$ such that $t_n\to +\infty$ as $n\to +\infty$, a sequence $(k_n)_n\subset\N$, and $(\varphi_n)_n\subset W^{1,\infty}((0,k_n)^2_{,1}; T_s(\mathcal{M}))$ such that $\varphi(x_{\alpha}, x_3) = 0$ for every $(x_{\alpha}, x_3) \in \partial (0,k_n)^2 \times \left(-\frac{1}{2},\frac{1}{2}\right)$ and such that for every $s\in \mathcal{M}$ and $\xi_\alpha\in\R^{3\times2}$
\begin{equation}
\label{recession}
    Tf^{0,\infty}_{hom}(s,\xi_\alpha)=\lim_{n\to+\infty}\frac{Tf^0_{hom}(s, t_n\xi_\alpha)}{t_n},
\end{equation}
and
\begin{equation*}
    \media_{(0,k_n)^2_{,1}} f(y_\alpha, y_3, t_n\xi_\alpha+t_n\nabla_\alpha\varphi_n|\nabla_3\varphi_n)dy\le Tf^0_{hom}(s,t_n\xi_\alpha)+\frac{1}{n}.
\end{equation*}
By (H2) and since $Tf_{hom}^0$ has linear growth, we have that
\begin{equation}
\label{asterisco}
    \int_{(0,k_n)^2_{,1}}|\nabla_{t_n}\varphi_n|dx\le C(1+|\xi_\alpha|),
\end{equation}
for some positive $C=C(\alpha,\beta)$. Since
\begin{equation*}
    \lim_{n\to+\infty}\frac{1}{t_n} \media_{(0,k_n)^2_{,1}} f(y_\alpha, y_3, t_n\xi_\alpha+t_n\nabla_\alpha\varphi_n|\nabla_3\varphi_n)dy = Tf^0_{hom}(s, \xi_\alpha),
\end{equation*}
then by (H4) and \eqref{recession} we have
\begin{align*}
    &Tf^0_{hom}(s,\xi_\alpha)-Tf^{0,\infty}_{hom}(s,\xi_\alpha)\\
    &\le \liminf_{n \to \infty}\left\{\media_{(0,k_n)^2_{,1}} \left|f(y_\alpha, y_3, t_n\xi_\alpha+t_n\nabla_\alpha\varphi_n|\nabla_3\varphi_n)-\frac{f(y_\alpha, y_3, t_n\xi_\alpha+t_n\nabla_\alpha\varphi_n|\nabla_3\varphi_n)}{t_n}\right|dy \right\}.
\end{align*}
Summing and subtracting $f^{\infty}(y_\alpha, y_3, t_n\xi_\alpha+t_n\nabla_\alpha\varphi_n|\nabla_3\varphi_n)$ we get
\begin{align*}
    &Tf^0_{hom}(s,\xi_\alpha)-Tf^{0,\infty}_{hom}(s,\xi_\alpha)\\
    &\le \liminf_{n \to \infty}\left\{\media_{(0,k_n)^2_{,1}} \left|f(y_\alpha, y_3, t_n\xi_\alpha+t_n\nabla_\alpha\varphi_n|\nabla_3\varphi_n)-f^\infty(y_\alpha, y_3, t_n\xi_\alpha+t_n\nabla_\alpha\varphi_n|\nabla_3\varphi_n)\right|dy \right.\\
    &\left. +\media_{(0,k_n)^2_{,1}} \left|f^{\infty}(y_\alpha, y_3, t_n\xi_\alpha+t_n\nabla_\alpha\varphi_n|\nabla_3\varphi_n)-\frac{f(y_\alpha, y_3, t_n\xi_\alpha+t_n\nabla_\alpha\varphi_n|\nabla_3\varphi_n)}{t_n}\right|dy \right\}.
\end{align*}
By (H4), \eqref{recession}, and the $1$-homogeneity of $f^\infty$ we get
\begin{align*}
    &Tf^0_{hom}(s,\xi_\alpha)-Tf^{0,\infty}_{hom}(s,\xi_\alpha)\\
    &\le \liminf_{n \to \infty}C\left\{\media_{(0,k_n)^2_{,1}} 1+|(\xi_\alpha+\nabla_\alpha \varphi_n|\nabla_3 \varphi_n)|^{1-q}dy\right.\\
    &\left.+\media_{(0,k_n)^2_{,1}} 1+t_n^{1-q}|(\xi_\alpha+\nabla_\alpha \varphi_n|\nabla_3 \varphi_n)|^{1-q}dy\right\}.
\end{align*}
Using H\"older inequality and \eqref{asterisco} we get
\begin{equation*}
    Tf^0_{hom}(s,\xi_\alpha)-Tf^{0,\infty}_{hom}(s,\xi_\alpha)\le C(1+|\xi_\alpha|^{1-q}).
\end{equation*}
Now we have to prove the inverse inequality. First, fix $k\in\N$ and $\varphi\in W^{1,\infty}((0,k)^2_{,1}; T_s(\mathcal{M}))$ such that $\varphi(x_{\alpha}, x_3) = 0$ for every $(x_{\alpha}, x_3) \in \partial (0,k)^2 \times \left(-\frac{1}{2},\frac{1}{2}\right)$, then from (H2) we have that for every $x\in (0,k)^2_{,1}$ and for every $t>0$
\begin{equation*}
    \frac{f(x, t(\xi_\alpha+\nabla_\alpha\varphi|\nabla_3\varphi))}{t}\le \beta(1+|\xi_\alpha+\nabla_\alpha\varphi|\nabla_3\varphi|),
\end{equation*}
so that $\frac{f(\cdot, t(\xi_\alpha+\nabla_\alpha\varphi|\nabla_3\varphi))}{t} \in L^1((0,k)^2_{,1})$. By Fatou's lemma we get
\begin{align*}
    Tf^{0,\infty}_{hom}(s,\xi_\alpha)&\le \limsup_{t\to \infty} \media_{(0,k)_{,1}} \frac{f(y, t(\xi_\alpha+\nabla_\alpha\varphi|\nabla_3\varphi))}{t}dy\\
    & \le \media_{(0,k)_{,1}} f^\infty(y, \xi_\alpha+\nabla_\alpha\varphi|\nabla_3\varphi)dy,
\end{align*}
which implies
\begin{equation}
\label{a}
    Tf^{0,\infty}_{hom}(s,\xi_\alpha)\le T(f^\infty)^{0}_{hom}(s,\xi_\alpha).
\end{equation}
Now, fix $\eta>0$ and let $\varphi\in W^{1,\infty}((0,k)^2_{,1}; T_s(\mathcal{M}))$ such that $\varphi(x_{\alpha}, x_3) = 0$ for every $(x_{\alpha}, x_3) \in \partial (0,k)^2 \times \left(-\frac{1}{2},\frac{1}{2}\right)$ such that
\begin{equation*}
    \media_{(0,k)^2_{,1}} f(y_\alpha, y_3, \xi_\alpha+\nabla_\alpha\varphi|\nabla_3\varphi)dy\le Tf^0_{hom}(s,\xi_\alpha)+\eta.
\end{equation*}
As before, by (H2) we have 
\begin{equation*}
        \media_{(0,k)^2_{,1}}\nabla \varphi\le C(1+|\xi_\alpha|),
\end{equation*}
so from \eqref{a} and (H4) we obtain
\begin{align*}
    &Tf^{0,\infty}_{\hom}(s,\xi_\alpha)- Tf^0_{hom}(s,\xi_\alpha)\le  T(f^{\infty})^0_{\hom}(s,\xi_\alpha)- Tf^0_{hom}(s,\xi_\alpha)\\
    & \le \media_{(0,k)^2_{,1}} \left|f^\infty(y_\alpha, y_3, \xi_\alpha+\nabla_\alpha\varphi|\nabla_3\varphi)-  f(y_\alpha, y_3, \xi_\alpha+\nabla_\alpha\varphi|\nabla_3\varphi)\right|dy+\eta\\
    & C\media_{(0,k)^2_{,1}} 1+|(\xi_\alpha+\nabla_\alpha\varphi|\nabla_3 \varphi)|^{1-q} dy +\eta
\end{align*}
By H\"older inequality we get
\begin{equation*}
    Tf^{0,\infty}_{\hom}(s,\xi_\alpha)- Tf^0_{hom}(s,\xi_\alpha)\le (1+|\xi|^{1-q}).
\end{equation*}
By the arbitrariness of $\eta$ this gives the desired inequality.
\end{proof}
\color{black}

\subsection{Jump part density}
Now we prove some properties for the density $\theta$ defined in \eqref{jumpterm}. To this aim, we introduce some further notations. In view of the fact that the results below are valid for any dimension in the setting of $N$ D$-(N-K)$D dimensional reduction, we present the main steps of the proof, which in dimension $2$ could be significantly simplified. In the remainder of this section there is an abuse of notation which will not generate contradiction with \eqref{jumpterm} in view the results we are going to show.\\

\noindent
Given $\nu = \{\nu_1, \nu_2\}$ an orthonormal basis of $\mathbb{R}^2$ and $(a,b) \in \mathcal{M} \times \mathcal{M},$ we denote by
\[
{\mathcal Q}_{\nu} := \{\lambda_1 \nu_1 + \lambda_2 \nu_2: \lambda_1, \lambda_2 \in (-1/2, 1/2)\}
\]
and, for $x_\alpha \in \mathbb{R}^2,$ we set
\begin{equation}
\label{x_nu}
    \|x_\alpha\|_{\nu, \infty} := \sup_{i = 1,2} |x_\alpha \cdot \nu_i|, \qquad x_{\nu} := x_\alpha \cdot \nu_1,  \qquad x' := x_\alpha \cdot \nu_2,
\end{equation}
so that $x_\alpha$ can be identified with the pair $(x_{\nu}, x')$. 
We also point out that this last definition of $\mathcal Q_\nu$ is a slight abuse of notation, indeed at the beginning of this section $Q_\nu$ has been introduced as the unit cube of $\R^2$, centered at the origin, and with two sides orthogonal to $\nu$, (with $\nu \in \mathbb S^1$). Here $\nu=\{\nu_1, \nu_2\}$ is an orthonormal basis of $\R^2$, but the notation $Q_\nu$ is still consistent since by definition $Q_{\nu_1}=Q_{\nu_2}=\mathcal{Q}_\nu$.
\\
Let $u_{a,b,\nu}: \R^2\times(-1/2,1/2) \rightarrow \mathcal{M}$ be the function defined by
\[
u_{a,b,\nu}
:= \left \{
\begin{array}{lll}
\!\!\!\!\!\! & a \qquad & \textnormal{if $x_{\nu} > 0$}\\
\!\!\!\!\!\! & b \qquad & \textnormal{if $x_{\nu} \le 0$}
\end{array}
\right.
\]
For every $t, h >0$,  we also introduce the following classes of functions
\begin{eqnarray}
\label{B set}
&& \mathcal{A}_t(a,b, \nu) := \{ \phi \in\ W^{1,1}((t {\cal Q}_{\nu})_{,1}; \mathcal{M}): \phi = u_{a,b, \nu} \text{ on }  \partial (t{\mathcal Q}_\nu)\times(-1/2,1/2)\},\nonumber\\
&& \mathcal{B}_{h}(a, b, \nu) := \Bigg \{u \in\ W^{1,1}({\mathcal Q}_{\nu,1}; \mathcal{M}): u(x) = \gamma \left (\frac{x_{\nu}}{h} \right )\nonumber \\
&& \text{ on  $\partial {\mathcal Q}_\nu\times(-1/2,1/2)$ for some $\gamma \in \mathcal{G}(a,b)$} \Bigg\},
\end{eqnarray}
where $\mathcal G$ is the set introduced in \eqref{geodesic}.
\begin{prop}
\label{BMBV3.2_3.3}
Let $f:\R^3\times\R^{3\times 3}\to \R$ be a Carath\'eodory function satisfying {\rm (H1)} $-$ {\rm (H4)}. For every $(a,b,\nu_1)\in\mathcal{M} \times \mathcal{M} \times \mathbb{S}^1,$ there exists
\begin{align}
\theta(a,b, \nu_1) &:= \lim_{t \rightarrow + \infty} \inf_{\phi} \left \{ \frac{1}{t} \int_{(t Q_{\nu})_{,1}} f^{\infty}(y, \nabla \phi) \, dy:  \phi \in \mathcal{A}_t(a, b, \nu)\right \}\label{unogammaA}\\
&=
\lim_{h \rightarrow 0} \inf_u  \left \{\int_{Q_{\nu,1}} f^{\infty}\left (\frac{x_{\alpha}}{h}, x_3, \nabla_h u \right ) dx: u \in \mathcal{B}_h(a, b, \nu) \right \}, \label{unogammaB}
\end{align}
where $\nu=(\nu_1,\nu_2)$ is any orthonormal base of $\R^2$ with first element $\nu_1$, i.e. the limit is dependent only on $\nu_1$.
\end{prop}
\color{black}
\begin{proof}
   The proof of Proposition \ref{BMBV3.2_3.3} follows in a similar way as in \cite{BMbv}, which in turn is very much inspired by \cite[Proposition 2.2]{BDFV} passing through the introduction of an appropriate rescaled surface energy density.
   We present also the main steps.
   In fact, first we prove that, for every $(a,b)\in \mathcal M\times \mathcal M$ and every orthonormal basis $\nu$, the limit in \eqref{unogammaB}  exists (see \cite[Proposition 3.3]{BMbv}), then we prove it also coincides with \eqref{unogammaA}, 
   i.e. a result analogous to \cite[Proposition 3.2]{BMbv} holds.
   It results that the limit in \eqref{unogammaB} exists whenever $\nu, \nu'$ are orthonormal rational bases, (i.e. there exists scalar $\gamma_1,\gamma'_1,\gamma_2,\gamma_2'\in \R\setminus \{0\}$ such that $\gamma_1 \nu_1, \gamma_1' \nu'_1, \gamma_2 \nu_2$ and $\gamma'_2 \nu'_2 \in \mathbb Z^2$), with equal first vector $\nu_1= \nu'_1$.
   Then, we claim that for every $\sigma > 0$ there exists $\delta>0$  (independent of $a$ and $b$) such that
if $\nu$ and $\nu'$ are two orthonormal bases of $\mathbb R^2$ with $|\nu_i -\nu_i'|< \delta$ for every $i = 1,2,$ then
\begin{align}
\label{label1}
\liminf_{h \to 0}I_h(\nu) - K\sigma \leq \liminf_{h \to 0}
I_h(\nu')\leq \limsup_{h \to 0}I_h(\nu') \leq \limsup_{h \to 0}
I_h(\nu) + K\sigma\end{align} 
where $K$ is a positive constant which only depends on $\mathcal M, \beta$ (see \eqref{H2}) and 
$$I_h(\nu):=I_h(a,b,\nu):=\inf \left\{\int_{\mathcal Q_{\nu,1}} f^{\infty}\left (\frac{x_{\alpha}}{h}, x_3, \nabla_h u \right ) dx: u \in \mathcal{B}_h(a, b, \nu) \right \}.$$

Let $\mathcal Q_{\nu,\eta} := (1 - \eta)\mathcal Q_\nu$ where $0 < \eta < 1$. Let $\sigma > 0$ be fixed and let
$0 < \eta < 1$ be such that
$\eta <\frac
{1}{34}$
and 
\begin{align}
\label{3.21}
\max\left\{\eta,
\frac{(1 -\eta)(1 - 2\eta)}{
(1 - 3\eta)}
- (1 -2\eta)\right\} <\sigma
\end{align}

Consider $\delta_0 > 0$ (that may be chosen so that $\delta_0 \leq \frac{\eta}{2
\sqrt{2}}$). Then take any $0 < \delta \leq \delta_0$
and any pair $\nu$ and $\nu'$ of orthonormal basis of $\mathbb R^2$ satisfying $|\nu_i -\nu'_i|\leq \delta$ for $i=1,2$ such that
one has
\begin{equation}\label{3.22}{\mathcal Q}_{\nu,3 \eta}\subset  \mathcal Q_{\nu',2 \eta}\subset {\mathcal Q}_{\nu, \eta}
\end{equation}
and $\{x \cdot \nu'_1=0\}\cap \partial {\mathcal Q}_{\nu,\eta}\subset  \{|x \cdot \nu_1| \leq 1/8\}$.
Given $h>0$  small, we consider $u_h \in \mathcal B_h(a, b, \nu')$
such that $$\int_{\mathcal Q_{\nu',1}}f^\infty\left( 
\frac{x_\alpha}{h},x_3,\nabla_h u_h\right)dx \leq I_h(\nu') + \sigma,$$
where $u_h(x)=u_h(x_\alpha, x_3) = \gamma_h\left(\frac{x_{\nu'}}{h}\right)$ for some $\gamma_h \in \mathcal{G}(a,b)$ and for $x \in \partial \mathcal Q_{\nu'}\times (-\tfrac{1}{2}, \tfrac{1}{2})$.
Then one can  construct $v_h \in \mathcal B_{(1-2\eta)h}(a, b, \nu)$ satisfying
the boundary condition $v_h(x) = \gamma_h \left(\frac{x_\nu}{(1 - 2\eta)h}\right)$ for $x\in \partial \mathcal Q_\nu\times (-1/2,1/2)$. Consider $F_\eta : \mathbb R^2 \to \mathbb R$,
\begin{equation*}F_\eta(x_\alpha) :=\left(
\frac{1 - 2\|x'\|_{\nu,\infty}}{\eta}\right)\frac{x_{\nu'}}{1 - 2\eta}
+
\left(
\frac{\eta- 1 + 2\|x'\|_{\nu,\infty}}
{\eta}\right)
\frac{x_\nu}{1-2 \eta},
\end{equation*}
and define
\begin{equation*}
v_h(x)=v_h(x_\alpha, x_3) :=
\left\{\begin{array}{ll}
u_h\left(\frac{x_\alpha}{1-2 \eta},x_3\right) &
\hbox{ if } x \in \mathcal Q_{\nu', 2 \eta}\times (-1/2,1/2), \\
\gamma_h\left(
\frac{x_{\nu'}}
{(1 - 2\eta)h}\right)
& \hbox{  if } x \in (\mathcal Q_{\nu,\eta}\setminus\mathcal Q_{\nu', 2 \eta})\times (-1/2,1/2)
\\
a &\hbox{ if } x \in (\mathcal Q_\nu \setminus \mathcal Q_{\nu,\eta})\times (-1/2,1/2) \hbox{ and } x_\nu \geq  \frac{1}{4}, \\
\gamma_h\left(\frac{F_\eta(x_\alpha)}{h}\right)
&\hbox{ if }  x \in A_\eta\cap (\mathcal Q_\nu \setminus \mathcal Q_{\nu,\eta})\times(-1/2,1/2),\\
b &\hbox{ if } x \in (\mathcal Q_\nu \setminus \mathcal Q_{\nu,\eta})\times (-1/2,1/2) \hbox{ and } x_\nu \leq \frac{-1}{4},
\end{array}
\right.
\end{equation*}
and $A_\eta:= \{x_\alpha : |x_\nu| \leq 1/4\}$.

One can check that $v_h$ is well defined for $h$ small enough and that $v_h \in \mathcal B_{(1-2\eta)h}(a, b, \nu).$
Therefore
\begin{align}
I_{(1-2\eta)h}(\nu) &\leq 
\int_{\mathcal Q_{\nu,1}}
f^\infty\left(\frac{x_\alpha}{(1-2\eta)h},x_3,\nabla_h v_h\right)
dx  \nonumber \\
&=
\int_{\mathcal Q_{{\nu',2\eta}, 1}}
f^\infty
\left(\frac{x_\alpha}{(1-2 \eta)h},x_3,
\nabla_h v_h\right)dx \nonumber\\ 
&+\int_{(\mathcal Q_{\nu,\eta}\setminus \mathcal Q_{\nu', 2 \eta})\times(-1/2,1/2)}
f^\infty\left(\frac{x_\alpha}{(1-2\eta)h},x_3, \nabla_h v_h\right)
dx\label{3.24B}\\
&+\int_{A_\eta\times(-1/2,1/2)}
f^\infty\left(\frac{x_\alpha}
{(1 -2\eta)h},x_3
,\nabla_h v_h\right)dx\nonumber \\
&=: I_1 + I_2 + I_3. \nonumber
\end{align}
We now estimate each of the three integrals in \eqref{3.24B}. First, we easily get that
\begin{equation}
    \label{3.24}
    I_1 = (1- 2\eta)\int_{\mathcal Q_{\nu'},1}
f^\infty\left(\frac{y_\alpha}{h},x_3,\nabla_h u_h\right)
dy_\alpha dx_3\leq
I_h(\nu')+\sigma.
\end{equation}
In view of \eqref{3.22} we have $\mathcal Q_{\nu,\eta} \subset (1 -\eta)(1 - 2\eta)(1 - 3\eta)^{-1}\mathcal Q_{\nu'}=: D_\eta$. Then we infer
from \eqref{recessiongrowth} and  Fubini’s theorem, one has
\begin{align}
\label{3.25}
I_2 &\leq \beta\int_{(D_\eta \setminus \mathcal Q_{\nu', 2 \eta})\times (-1/2,1/2)}
|\nabla_h v_h| dx \\
&= \frac{\beta}{(1 - 2\eta)h}\int_{((D_\eta\setminus \mathcal Q_{\nu',2 \eta})\cap
\{|x_\nu|\leq{(1-2\eta)h}/2)\}\times(-1/2,1/2)}
\left|\dot{\gamma}_h\left(\frac{x_\nu}{(1-2\eta)h}\right)\right|dx_\alpha d x_3\nonumber\\
&= \beta\mathcal{H}^1
((D_\eta \setminus \mathcal Q_{\nu',2\eta})\cap\{x_{\nu'}= 0\})\frac{1}{
(1-2\eta)h}
\int^{(1-2\eta)h/2}_{-(1-2\eta)h/2}\left|\dot{\gamma_h}\left(\frac{t}{(1-2\eta)h}\right)\right|dt\nonumber\\
&=\beta{\rm dist}_\mathcal{M}(a,b)\left(  \frac{(1 -\eta)(1 - 2\eta)}{
(1 - 3\eta)}
- (1 -2\eta) \right).\nonumber
\end{align}
Now we estimate $I_3$. In order to do so, we observe that from \eqref{3.22} we get
\begin{equation}
    \label{3.26}
    ||\nabla_\alpha F_\eta(x_\alpha)||_{L^\infty(A_\eta,\R^2)}\le C
\end{equation}
for some absolute constant $C>0$, and 
\begin{equation}
    \label{3.27}
    |\nabla_\alpha F_\eta(x_\alpha)\cdot \nu_1|\ge 1,
\end{equation}
for a.e. $x_\alpha\in A_\eta$. Combining \eqref{recessiongrowth}, \eqref{3.26}, and \eqref{3.27} we get
\begin{align}
    I_3&\le \beta \int_{A_\eta}|\nabla_h v_h|dx_\alpha\le \frac{C\beta}{h}\int_{A_\eta}\left|\dot{\gamma_h}\left(\frac{F_\eta(x_\alpha)}{h}\right)\right|dx_\alpha\nonumber\\
    &\le \frac{C\beta}{h}\int_{A_\eta}\left|\dot{\gamma_h}\left(\frac{F_\eta(x_\alpha)}{h}\right)\right||\nabla_\alpha F_\eta(x_\alpha)\cdot \nu_1|dx_\alpha\nonumber\\
    &= C\beta\int_{A'_\eta}\frac{1}{h}\int_{-1/4}^{1/4}\left|\dot{\gamma_h}\left(\frac{F_\eta(t\nu_1+ x')}{h}\right)\right||\nabla_\alpha F_\eta(t\nu_1+ x')\cdot \nu_1|dt\,d\mathcal{H}^1(x'), \nonumber
\end{align}
where we used Fubini's theorem in the last equality and $A'_\eta:=A_\eta\cap\{x_\nu=0\}$. By the change of variable $s=(1/h)F_\eta(t\nu_1+ x')$ it follows that for $\mathcal{H}^1$-a.e. $x'\in A_\eta$
\begin{equation*}
\int_{-1/4}^{1/4}\left|\dot{\gamma_h}\left(\frac{F_\eta(t\nu_1+ x')}{h}\right)\right||\nabla_\alpha F_\eta(t\nu_1+ x')\cdot \nu_1|dt\le \int_\R|\dot{\gamma_h}(s)| ds={\rm dist}_\mathcal{M}(a,b).
\end{equation*}
Combining the previous two inequalities we get
\begin{equation}
    \label{3.28}
    I_3\le C\beta \mathcal{H}^1(A'_\eta){\rm dist}_\mathcal{M}(a,b)=C\beta \eta {\rm dist}_\mathcal{M}(a,b)
\end{equation}
By \eqref{3.24}, \eqref{3.25}, \eqref{3.28}, and \eqref{3.21} we get
\begin{equation*}
    I_{(1-2\eta)h}(\nu)\le I_h(\nu')+K\sigma,
\end{equation*}
with $K=1+\beta\delta(1+C)$, $\delta$ diameter of $\mathcal{M}$ and $C$ the constant from \eqref{3.26}. Passing to the limit $h\to 0$ and by the arbitrariness of $\nu$ and $\nu'$ we get \eqref{label1}. Arguing as in \cite[Proposition 2.2, Step 4]{BDFV} we get that for any orthonormal basis $\nu$ and $\nu'$ with first element $\nu_1\in \mathbb{S}^1$, we get the existence of the limit of $I_h$ for every $\nu$ and 
\begin{equation}\label{existlimIh}
    \lim_{h\to 0}I_h(\nu)=\lim_{h\to 0}I_h(\nu').
\end{equation}
\\

\noindent
Now we prove the equality between \eqref{unogammaA} and \eqref{unogammaB}. Fix $h>0$ and an orthonormal basis $\nu=(\nu_1,\nu_2)$ of $\R^2$. Let
\begin{equation*}
    J_h(\nu)=J_h(a,b, \nu):=\inf\left\{\int_{\mathcal Q_\nu\times(-1/2,1/2)}f^\infty\left(\frac{x_\alpha}{h},x_3,\nabla_h u\right)dx: u\in\mathcal{A}_h(a,b,\nu)\right\},
\end{equation*}
by a change of variable we have also that
\begin{equation*}
    J_h(\nu)=\inf\left\{h^2\int_{\frac{1}{h}\mathcal Q_\nu\times(-1/2,1/2)}f^\infty\left({x_\alpha},x_3,\nabla u\right)dx: u\in\mathcal{A}_{1/h}(a,b,\nu)\right\},
\end{equation*}
We want to prove that
\begin{equation}
\label{BMBV3.29}
    \lim_{h\to 0}J_h(\nu)=\lim_{h\to 0}I_h(\nu).
\end{equation}
For $h$ small enough, set $\tilde{h}:=h/(1-h)$ and consider $u_{\tilde{h}}\in \mathcal{B}_{\tilde{h}}(a,b, \nu)$ such that
\begin{equation*}
\int_{\mathcal Q_\nu\times(-1/2,1/2)}f^\infty\left(\frac{x_\alpha}{\tilde{h}},x_3,\nabla_{\tilde{h}} u_{\tilde{h}}\right)dx\le I_{\tilde{h}}(\nu)+h,
\end{equation*}
with $u_{\tilde{h}}(x)=\gamma_{\tilde{h}}(x_\nu/\tilde{h})$ in $\partial\mathcal Q_\nu\times (-1/2,1/2)$, for some $\gamma_{\tilde{h}}\in\mathcal{G}(a,b)$. Now, for every $x\in\mathcal Q_\nu\times (-1/2,1/2)$ we define
\begin{equation*}
    v_{h}(x):=
    \begin{cases}
        u_{\tilde{h}}(\frac{x_\alpha}{1-h},x_3) \quad \text{if }x\in \mathcal{Q}_{\nu,h}\\
        \gamma_{\tilde{h}}\left(\frac{x_\nu}{1-2||x'||_{ \nu, \infty}}\right) \quad\text{otherwise}
    \end{cases}.
\end{equation*}
It follows that
\begin{align}
    J_h(\nu)&\le\int_{\mathcal Q_{\nu}\times(-1/2,1/2)}f^\infty\left(\frac{x_\alpha}{{h}},x_3,\nabla_{{h}} v_h\right)dx \label{Jhest}\\
    &\le \int_{\mathcal Q_{\nu, h}\times(-1/2,1/2)}f^\infty\left(\frac{x_\alpha}{h},x_3,\nabla_{h} v_h\right)dx +\int_{(\mathcal Q_\nu\setminus \mathcal Q_{\nu,h})\times(-1/2,1/2)}f^\infty\left(\frac{x_\alpha}{h},x_3,\nabla_{h} v_h\right)dx \nonumber\\
    & = I_1+I_2. \nonumber
\end{align}
Now we estimate $I_1$ and $I_2$. For $I_1$ we have that by construction

\begin{equation}
\label{3.30}
    I_1=(1-h)\int_{\mathcal Q_\nu\times(-1/2,1/2)} f^\infty\left(\frac{x_\alpha}{\tilde{h}},x_3, \nabla_{\tilde{h}}u_{\tilde{h}}\right)dx\le (1-h)(I_{\tilde{h}}(\nu)+h).
\end{equation}
Now we turn to $I_2$.
\begin{align*}
    I_2&\le \beta\int_{(\mathcal Q_\nu\setminus\mathcal Q_{\nu,h})\times(-1/2,1/2)}\left|\dot{\gamma}_{\tilde{h}}\left(\frac{x_\nu}{1-2||x'||_{\nu,\infty}}\right)\right|\left( \frac{1}{1-2||x'||_{\nu,\infty}} +\frac{|x_\nu||\nabla_\alpha (||x'||_{\nu,\infty})|}{(1-2||x'||_{\nu,\infty})^2}\right)dx\\
    &\le 2\beta \int_{(\mathcal Q_\nu\setminus\mathcal Q_{\nu,h})\cap\{x_\nu\le (1-2||x'||_{\nu,\infty})/2\}}\left|\dot{\gamma}_{\tilde{h}}\left(\frac{x_\nu}{1-2||x'||_{\nu,\infty}}\right)\right|\left( \frac{1}{1-2||x'||_{\nu,\infty}}\right)dx_\alpha
\end{align*}
where, in order, we used the growth condition of $f^\infty$, the fact that $\dot{\gamma}_{\tilde{h}}\left(\frac{x_\nu}{1-2||x'||_{\nu,\infty}}\right)=0$ on the complementary set of $\{x_\nu\le (1-2||x'||_{\nu,\infty})/2\}$, and the fact that $||\nabla_\alpha(||x'||_{\nu,\infty})||_{L^\infty(\mathcal{Q}_\nu,\R^2)}\le 1$. Denoting $\mathcal{Q}'_\nu:=\mathcal{Q}_\nu\cap\{x_\nu=0\}$ and $\mathcal{Q}'_{\nu,h}:=\mathcal{Q}_{\nu,h}\cap\{x_\nu=0\}$, then by Fubini's Theorem we get
\begin{align}
\label{3.31}
    I_2&\le 2\beta \int_{(\mathcal Q'_\nu\setminus\mathcal Q'_{\nu,h})}\int_{-(1-2||x'||_{\nu,\infty})/2}^{(1-2||x'||_{\nu,\infty})/2}\left|\dot{\gamma}_{\tilde{h}}\left(\frac{t}{1-2||x'||_{\nu,\infty}}\right)\right|\left( \frac{1}{1-2||x'||_{\nu,\infty}}\right)dtdx'\nonumber\\
    & \le 2\beta{\rm dist}_{\mathcal{M}}(a,b)\mathcal{H}^1(\mathcal Q'_\nu\setminus\mathcal Q'_{\nu,h})\le 2\beta{\rm dist}_{\mathcal{M}}(a,b)h,
\end{align}
and so from \eqref{Jhest}, \eqref{3.30}, \eqref{3.31}  and \eqref{existlimIh}, we infer that
\begin{equation*}
    \limsup_{h\to 0}J_h(\nu)\le\lim_{h\to 0}I_h(\nu).
\end{equation*}
On the other hand, if we fix $h>0$ small enough and  $\tilde{u}_h\in\mathcal{A}_1(a,b,\nu)$ such that
\begin{equation*}
    \int_{\mathcal Q_{\nu}\times(-1/2,1/2)}f^\infty\left(\frac{x_\alpha}{{h}},x_3,\nabla_{{h}} \tilde{u}_h\right)dx\le J_h(\nu)+h,
\end{equation*}
and we denote for $x\in \mathcal{Q}_\nu\times(-1/2,1/2)$
\begin{equation*}
    w_h(x):=
    \begin{cases}
        \tilde{u}_h\left(\frac{x_\alpha}{1-h},x_3\right) \quad \text{if } x_\alpha\in \mathcal{Q}_{\nu, h}\\
        \gamma\left(\frac{x_\nu}{(2||x'||_{\nu,\infty}-1+h)(1-h)}\right) \quad \text{otherwise}
    \end{cases}
\end{equation*}
with $\gamma\in\mathcal{G}(a,b)$, then it is easy to check that $w_h\in \mathcal{B}_{(1-h)h}(a,b,\nu)$. Arguing as in the previous estimates we get that
\begin{align*}
    I_{(1-h)h}(\nu)\le (1-h)(J_h(\nu)+h)+2\beta h{\rm dist}_{\mathcal{M}}(a,b),
\end{align*}
which leads to 
\begin{equation*}
    \liminf_{h\to 0} J_h(\nu)\ge \lim_{h\to 0} I_{h}(\nu),
\end{equation*}
and this concludes the proof.
\end{proof}

\color{black}

\color{black}

\begin{prop}
\label{BMbvProp3.4}
    The function $\theta$ defined is \eqref{jumpterm} is continuous on $\mathcal{M}\times\mathcal{M}\times\mathbb{S}^1$ and there exist constants $C_1,C_2>0$ such that
    \begin{align}
        &|\theta(a_1,b_1, \nu_1)-\theta(a_2,b_2, \nu_1)|\le C_1(|a_1-a_2|+|b_1-b_2|),\label{cont}\\
        &\theta(a_1,b_1,\nu_1)\le C_2|a_1-b_1|\label{dist},
    \end{align}
    for every $a_1,b_1,a_2,b_2\in\mathcal{M}$ and $\nu_1\in \mathbb{S}^1$.
\end{prop}
\begin{proof}
    It follows in a similar way to \cite[Proposition 3.4]{BMbv}, taking into account similar modifications as we did in Proposition \ref{BMBV3.2_3.3}. For the reader's convenience, we present a short proof. For shortness sake, we use the same notation as in the previous proof. From the previous proof it is clear that $\theta (a,b,\cdot)$ is uniformly continuous on $\mathbb{S}^1$ so it is enough to prove \eqref{cont} in order to prove the continuity of $\theta$ on $\mathcal{M}\times\mathcal{M}\times\mathbb{S}^1$, and thus we start from there. Fix $\nu_1\in\mathbb{S}^1$ and let $\nu=(\nu_1,\nu_2)$ be an orthonormal base of $\R^2$. For $h >0$, we denote $\tilde{h}:=h/(1-h)$ and we consider $\gamma_{\tilde{h}}\in \mathcal{G}(a_1,b_1)$ and $u_{\tilde{h}}\in \mathcal B_{\tilde{h}}(a_1,b_1,\nu)$ such that $u_{\tilde{h}}(x)=\gamma_{\tilde{h}}(\frac{x_\nu}{\tilde{h}})$ on $(\partial \mathcal Q_\nu)_1$ and
    \begin{equation*}
        \int_{ \mathcal Q_{\nu,1}}f^\infty\left(\frac{x_\alpha}{\tilde{h}},x_3,\nabla_{\tilde{h}} u_{\tilde{h}}\right)dx\le I_{\tilde{h}}(a_1,b_1,\nu)+h.
    \end{equation*}
    As in the previous proof, now we construct a function $v_h\in\mathcal{A}_{1}(a_2,b_2,\nu)$ starting from $u_{\tilde{h}}$. Fix $\gamma_a\in\mathcal{G}(a_2,a_1)$, $\gamma_b\in\mathcal{G}(b_2,b_1)$, and define
    \begin{equation*}
        v_h(x):=
        \begin{cases}
        &u_{\tilde{h}}\left(\frac{x_\alpha}{1-h},x_3\right) \quad x\in \mathcal Q_{\nu,h}\times(-1/2,1/2)\\

        &\gamma_{\tilde{h}}\left(\frac{x_\nu}{1-2||x'||_{\nu,\infty}}\right) \quad x\in A_1\times(-1/2,1/2)\\

        &\gamma_a\left(\frac{2||x_\alpha||_{\nu,\infty}-1}{h}+\frac{1}{2}\right)\quad x\in A_2\times(-1/2,1/2)\\

        &\gamma_b\left(\frac{2||x_\alpha||_{\nu,\infty}-1}{h}+\frac{1}{2}\right)\quad x\in A_3\times(-1/2,1/2)\\
        
        &\gamma_a\left(\frac{2||x'||_{\nu,\infty}-1}{2x_\nu}+\frac{1}{2}\right)\quad x\in A_4\times(-1/2,1/2)\\

        &\gamma_b\left(\frac{1-2||x'||_{\nu,\infty}}{2x_\nu}+\frac{1}{2}\right)\quad x\in A_5\times(-1/2,1/2)
        \end{cases}
    \end{equation*}
with
\begin{align*}
    &A_1:=\left\{\frac{1-h}{2}\le||x'||_{\nu,\infty}<\frac{1}{2}\text{ and } |x_\nu|\le-||x'||_{\nu,\infty}+\frac{1}{2}\right\}\\
    &A_2:=(\mathcal Q_\nu\setminus \mathcal Q_{\nu,h})\cap\left\{x_\nu \ge h/2 \right\}\\
    &A_3:=(\mathcal Q_\nu\setminus \mathcal Q_{\nu,h})\cap\left\{x_\nu \le -h/2 \right\}\\
    &A_4:=\left\{0<x_\nu\le\frac{h}{2}, \frac{1}{2}-x_\nu\le||x'||_{  \nu,\infty}< \frac{1}{2}\right\}\\
    &A_5:=\left\{-\frac{h}{2}<x_\nu\le 0, \frac{1}{2}+x_\nu\le||x'||_{  \nu,\infty}< \frac{1}{2}\right\}.
\end{align*}
Since by construction $v_h\in\mathcal{A}_{1}(a_2,b_2,\nu)$, then
\begin{equation*}
 J_h(a_2,b_2,\nu)\le\int_{ \mathcal Q_{\nu,1}}f^\infty\left(\frac{x_\alpha}{\tilde{h}},x_3,\nabla_{{h}} v_{{h}}\right)dx.
\end{equation*}
By the same argument as in the proof of Proposition \ref{BMBV3.2_3.3}, we get
\begin{equation*}
    \int_{ \mathcal Q_{\nu,1}}f^\infty\left(\frac{x_\alpha}{h},x_3,\nabla_{{h}} v_{{h}}\right)dx\le I_{\tilde{h}}(a_1,b_1,\nu)+h,
\end{equation*}
and
\begin{equation*}
    \int_{A_1\times(-1/2,1/2)}f^\infty\left(\frac{x_\alpha}{h},x_3,\nabla_{{h}} v_{{h}}\right)dx\le C{\rm dist}_{\mathcal{M}}(a_1,b_1)h,
\end{equation*}
where ${\rm dist}_{\mathcal{M}}$ denotes the geodesic distance on $\mathcal{M}$. Now we estimate the integrals over $A_2$ and $A_4$. The integrals on $A_3$ and $A_5$ follow in a similar way. We start with $A_2$. First, we define 
\begin{equation*}
F_h(x_\alpha):=\frac{2||x_\alpha||_{\nu,\infty}-1}{h}+\frac{1}{2}
\end{equation*}
From the linear growth of $f^\infty$, Fubini's Theorem, the fact that $A_2\subset F_h^{-1}([-1/2,1/2))$, the Coarea formula, and the fact that $v_h$ does not depends on $x_3$ on $A_2$ we get
\begin{align*}
\int_{A_2\times(-1/2,1/2)}f^\infty\left(\frac{x_\alpha}{h},x_3,\nabla_{{h}} v_{{h}}\right)dx&\le \beta\int_{A_2}|\dot{\gamma}_a(F_h(x_\alpha))||\nabla_\alpha F_h(x_\alpha)|dx_\alpha\\
& \le \beta\int_{F_h^{-1}([-1/2,1/2))}|\dot{\gamma}_a(F_h(x_\alpha))||\nabla_\alpha F_h(x_\alpha)|dx_\alpha\\
& \le \beta \int_{(-1/2,1/2)}|\dot{\gamma}_a(t)| \mathcal{H}^1(F_h^{-1}(t))dt
\end{align*}
Since $F_h^{-1}(t)=\partial \mathcal Q_{\nu, \frac{h(1-2t)}{2}}$ for every $t\in (-1/2,1/2)$, then it holds $\mathcal{H}^1(F_h^{-1}(t))\le \mathcal{H}^1(\partial \mathcal Q_\nu)=2$, which gives us
\begin{equation*}
\int_{A_2\times(-1/2,1/2)}f^\infty\left(\frac{x_\alpha}{h},x_3,\nabla_{{h}} v_{{h}}\right)dx\le 2\beta{\rm dist}_{\mathcal{M}}(a_1,a_2)
\end{equation*}
Now we estimate the integral on $A_4$. We define
\begin{equation*}
    G_h(x_\alpha):=\frac{2||x'||_{\nu,\infty}-1}{2x_\nu}+\frac{1}{2}.
\end{equation*}
As before, by the linear growth of $f^\infty$ and Fubini's Theorem we get
\begin{align*}
    &\int_{A_4\times(-1/2,1/2)}f^\infty\left(\frac{x_\alpha}{h},x_3,\nabla_{{h}} v_{{h}}\right)dx\\
    &\le \beta\int_{(0,h/2)}\int_{G^{-1}(\cdot,x_\nu)([-1/2,1/2))}|\dot{\gamma_a}(G(x',x_\nu))||\nabla_\alpha G(x',x_\nu)|dx'dx_\nu
\end{align*}
Since $|\nabla_{x'} G(x_\alpha)|= 1/x_\nu$ and $|\nabla_{x_\nu} G(x_\alpha)|\le 1/x_\nu$ for a.e. $x_\alpha\in A_4$ we get that $|\nabla_\alpha G(x_\alpha)|\le 2 |\nabla_{x'} G(x_\alpha)|$ and thus
\begin{align*}
    &\int_{A_4\times(-1/2,1/2)}f^\infty\left(\frac{x_\alpha}{h},x_3,\nabla_{{h}} v_{{h}}\right)dx\\
    &\le 2\beta\int_{(0,h/2)}\int_{G^{-1}(\cdot,x_\nu)([-1/2,1/2))}|\dot{\gamma_a}(G(x',x_\nu))||\nabla_{x'} G(x',x_\nu)|dx'dx_\nu
\end{align*}
By construction, $G(\cdot, x_\nu)$ is Lipschitz for every $x_\nu\in (0,h/2)$ and so by the Coarea formula we get again that
\begin{align*}
    &\int_{A_4\times(-1/2,1/2)}f^\infty\left(\frac{x_\alpha}{h},x_3,\nabla_{{h}} v_{{h}}\right)dx\\
    &\le 2\beta\int_{(0,h/2)}\int_{(-1/2,1/2)}|\dot{\gamma_a}(t)|\mathcal{H}^0(\{x':G(x',x_\nu)=t\})dtdx_\nu\\
    &\le C\,h\,{\rm dist}_{\mathcal{M}}(a_1,a_2)
\end{align*}
where we used the fact that $\mathcal{H}^0(\{x':G(x',x_\nu)=t\})=1$ for every $t\in (-1/2,1/2)$.
Estimating the integrals over $A_3$ and $A_5$ in a similar way we get
\begin{equation*}
    J_h(a_2,b_2)\le \int_{\mathcal Q_{\nu,1}}f^\infty\left(\frac{x_\alpha}{h},x_3,\nabla_{{h}} v_{{h}}\right)dx\le I_h(a_1,b_1) + C(h + {\rm dist}_{\mathcal{M}}(a_1,a_2)+{\rm dist}_{\mathcal{M}}(b_1,b_2)).
\end{equation*}
Passing to the limit for $h\to 0$ we get
\begin{equation*}
    \theta(a_1,b_1,\nu)\le \theta(a_2,b_2,\nu) +C({\rm dist}_{\mathcal{M}}(a_1,a_2)+{\rm dist}_{\mathcal{M}}(b_1,b_2)))
\end{equation*}
since ${\rm dist}_{\mathcal{M}}$ is equivalent to the Euclidian distance, and by the arbitrariness of $a_1,a_2,b_1,$ and $b_2$ in $\mathcal{M}$ we get \eqref{cont}.\\

\noindent
Now we prove \eqref{dist}. Fix an orthonormal basis $\nu=(\nu_1,\nu_2)$ of $\R^2$ and let $\gamma\in \mathcal{G}(a_1,b_1)$. Consider $u_h(x)=\gamma(x_\alpha/h)$. By construction, $u_h\in \mathcal B_h(a_1,b_1,\nu)$. Moreover, since $\lim_{h\to 0} J_h(\nu)=\lim_{h\to 0} I_h(\nu)$ it holds
\begin{align*}
    \theta(a_1,b_1,\nu)&\le \liminf_{h\to 0} \int_{\mathcal Q_{\nu,1}}f^\infty\left(\frac{x_\alpha}{h},x_3,\nabla_{{h}} u_{{h}}\right)dx\\
    &\le \liminf_{h\to 0} \frac{\beta}{h}\int_{\mathcal Q_\nu}\left|\dot{\gamma}\left(\frac{x_\alpha\cdot\nu}{h}\right)\right|dx_\alpha\\
    &\le \beta{\rm dist}_{\mathcal{M}}(a_1,b_1,\nu).
\end{align*}
Again, by the equivalence of ${\rm dist}_{\mathcal{M}}$ with the Euclidian distance we get \eqref{dist}.
\end{proof}

\begin{prop}
\label{BMBVprop_4.1}
    There exists a function $K: \mathcal{M} \times \mathcal{M} \times \mathbb{S}^{1} \rightarrow [0, + \infty)$ continuous in the last variable and such that
\\
{\rm (i)} $K(a, b, \nu) = K (b, a, - \nu)$ for every $(a, b, \nu) \in  \mathcal{M} \times \mathcal{M} \times \mathbb{S}^{1},$
\\
{\rm (ii)} for every finite subset $T$ of $\mathcal{M},$
\begin{equation}
\label{BMBV_prop4.1}
    I^0(u, S) = \int_S K (u^+, u^-, \nu_u) \, d \mathcal{H}^{1},
\end{equation}
for every $u \in BV(\omega; T)$ and every Borel subset $S$ of $\omega \cap S_u.$
\end{prop}
\begin{proof}
Let $T$ be a finite subset of $\mathcal M$ and $\omega\subset\R^2$ open and bounded. For every $A \in \mathcal{A}(\omega)$ and $u \in BV(\omega; T)$ we define
\[
G_T(u, A) := I^0(u, S_u \cap A),
\]
where $I^0(u, S_u \cap A)$ is defined as in \eqref{IoBV}.

We aim to show that the assumptions of \cite[Theorem 3.1]{AB} are satisfied by $G_T: BV(\omega; T) \times \mathcal{A}(\omega) \to [0, +\infty)$, namely:

\begin{enumerate}
    \item $0 \leq G_T(u, A) \leq \Lambda \mathcal{H}^{n-1}(A \cap S_u)$ with a fixed constant $\Lambda \in \mathbb{R}$;
    \item $G_T(u, \cdot)$ is the restriction to $\mathcal{A}(\omega)$ of a Borel measure on $\omega$;
    \item $G_T(u, A) = G_T(v, A)$ if $u = v$ a.e. in $A$;
    \item If $u_h \to u$ a.e. in $A$, then
    \[
    G_T(u, A) \leq \liminf_{h \to +\infty} G_T(u_h, A);
    \]
    \item For every $y \in \mathbb{R}^2$, we have
    \[
    G_T(\tau_y u, \tau_y A) = G_T(u, A),
    \]
    where $(\tau_y u)(x) = u(x - y)$ and $\tau_y A = A + y$, provided $\tau_y A \subset \Omega$.
\end{enumerate}

Property (1) and (3) follows immediately from the definition of $I^0$. Property (2), follows from the fact the Borel measure $I^0(u, \cdot)$ on $\omega$ (see \ref{LocalizationBV}), restricted to $S_u$, is an extension of $G_T(u, \cdot)$.

For (4), if $u_h \to u$ a.e. on $A$, then the sequence $(u_h)$ is equibounded in $A$. It follows that $(u_h)_h$ converges to $u$ in $L^1(A;\mathcal{M})$.  It not difficult to see that both $u_h$ and $u$ can be extended to $\Omega$ by taking the extension constant in the third variable, and so that $u_h \to u$ a.e. on $A_{,1}$ and $(u_h)_h$ converges to $u$ in $L^1(A_{,1};\mathcal{M})$. For every open set $E \subset A$ with $S_u \cap A \subset E$, we have
\[
I^0(u, E) \leq \liminf_{h \to \infty} I^0(u_h, E).
\]
Moreover, from \eqref{LocalizationBV}, (H2), and the equiboundedness of $(u_h)_h$ we get
\[
I^0(u_h, E) = I^0(u_h, S_{u_h} \cap E) + I^0(u_h, E \setminus S_{u_h}) \leq I^0(u_h, S_{u_h} \cap A) + C |E|.
\]
Taking the infimum over $E \supset S_u \cap A$, we obtain
\[
I^0(u, S_u \cap A) \leq \liminf_{h \to \infty} I^0(u_h, S_{u_h} \cap A).
\]
Finally, property (5) follows from the translation invariance of $I^h$ (see (H1) and \eqref{functional}) with respect to the $x_\alpha$-variable.
We can now apply Theorem 3.1 in \cite{AB}, from which follows the existence of a continuous function $K_\omega^T : \omega \times T \times T \times \mathbb{S}^{1} \to [0, +\infty[$ such that $K_\omega^T(x_\alpha, a, b, \nu)=K_\omega^T(x_\alpha, a, b, -\nu)$ and
\[
I^0(u, S_u \cap A) = G_T(u, A) = \int_{S_u \cap A} K_\omega^T(x_\alpha, u^+, u^-, \nu_u) \, d\mathcal{H}^{1}
\]
for every $u \in BV(\omega; T)$ and $A \in \mathcal{A}(\omega)$. For $x_\alpha, x_0\in\mathbb{R}^2$, $a,b \in \mathcal{M}$ and $\nu \in \mathbb{S}^1$ define 
\[
u^{a,b}_{x_0,\nu_1}(x) :=
\begin{cases}
	a & \text{if } (x_\alpha-x_0) \cdot \nu_1 \geq 0, \\
	b & \text{if } (x_\alpha-x_0) \cdot \nu_1 < 0
\end{cases}, \quad \Pi_{x_0, \nu_1} := \{ x_\alpha \in \mathbb{R}^2 : (x_\alpha-x_0) \cdot \nu_1 = 0 \}.
\]
The continuity of $K_\omega^T$ implies that
\begin{equation*}
	K_\omega^T(x_0,a,b, \nu)=\lim_{\rho \to 0} \frac{I^0(u^{a,b}_{x_0,\nu_1}(x), B_{\rho}(x_0)\cap\Pi_{x_0, \nu_1} )}{\mathcal{H}^1( B_{\rho}(x_0)\cap\Pi_{x_0, \nu_1})},
\end{equation*}
for every $(x_0,a,b,\nu)\in \omega\times T\times T\times \mathbb{S}^1$. By the arbitrariness of $T$ and $\omega$ it follows that $K_\omega^T$ is independent from $T$ and $\omega$, i.e. $K_\omega^T=K$. By a blow-up argument it is possible to show that $K$ does not depend from the $x_\alpha$-variable. Since $I^0(u, \cdot)$ restricted to $S_u$ is a regular Borel measure, it follows the integral representation on the Borel subsets of $S_u$ and this concludes the proof.
\end{proof}

\section{Proof of Theorem \ref{BVcase}}
\label{BV}
\color{black}
\noindent
In this section we prove Theorem \ref{BVcase}. In particular, we assume $f$ to satisfy hypothesis (H1), (H2), (H3) and (H4).\\

\noindent
The first step in order to prove Theorem \ref{BVcase} is to localize $I^0$, with
\begin{equation}
    \label{IoBV}
    I^0(u,A):=\inf\left\{\liminf_{n\to+\infty}I^{h_n}(u_n,A) : u_n \to u\,\text{ in $L^1(A_{,1};\mathcal M)$}\right\},
\end{equation}
for any  $u\in BV(\omega; \mathcal{M})$ and $A\in\mathcal{A}(\omega)$.

\begin{lemma}[Localization]
    \label{LocalizationBV}
    Given $u\in BV(\omega;\mathcal{M})$ the set function $I^0(u,\cdot):\mathcal{A}(\omega)\to \R$ is the restriction to $\mathcal{A}(\omega)$ of a Radon measure absolutely continuous with respect to $\mathcal{L}^2+|Du|$.
\end{lemma}

\begin{proof}
Fix $u\in BV(\omega; \mathcal{M})$ and $A\in\mathcal{A}(\omega)$. By \cite[Theorem 3.9]{AFP} there exists a sequence $(u'_n)_n\subset W^{1,1}(A;\R^3)\cap C^\infty(A;\R^3)$ such that $u'_n\to u$ in $L^1(A;\R^3)$ and $||\nabla_{\alpha}u'_n||_{L^1(A;\R^3)}\to |D_\alpha u|(A)$. So, if we consider $u_n(x):=u'_n(x_\alpha)$, we have that $(u_n)_n\subset W^{1,1}(A_{,1};\R^3)\cap C^\infty(A_{,1};\R^3) $ and $||\nabla_{h_n}u_n||_{L^1(A_{,1};\R^3)}\to |D_\alpha u|(A)$.
By construction, $u_n(x)\in co(\mathcal{M})$ for a.e. $x\in A_{,1}$ and every $n\in \N$, where $co(\mathcal{M})$ stands for the convex envelope of $\mathcal{M}$. Following the argument of \cite[Proposition 2.1]{BMbv}, there exists a sequence $(w_n)_n\subset W^{1,1}(A_{,1};\mathcal{M})$ satisfying
\begin{equation*}
    \int_{A_{,1}} |\nabla_{h_n} w_n|dx\le C^*\int_{A_{,1}} |\nabla_{h_n} u_n|dx,
\end{equation*}
for some constant $C^*$ depending only on $\mathcal{M}$. Moreover, we have $w_n\to u$ in $L^1(A_{,1}; \R^3)$ and by the growth condition $(H2)$ we have 
\begin{equation*}
    I^0(u, A)\le \beta(\mathcal{L}^2(A)+C^*|Du|(A)).
\end{equation*}
Now we just have to prove that
\begin{equation*}
    I^0(u,A)\le I^0(u,B)+I^0(u, A\setminus \Bar{C}),
\end{equation*}
for any $A, B, C\in\mathcal{A}(\omega)$ with  $\Bar{C}\subset B\subset A$. The proof of the previous inequality follows from a standard slicing argument similar to \cite[Lemma 3.2]{LTZ1}.
\end{proof}

Now we prove the limsup inequality.

\begin{prop}\label{limsupbv}
    For every $u\in BV(\omega;\mathcal{M})$ it holds
    \begin{equation*}
        I^0(u)\le I(u),
    \end{equation*}
    with $I$ defined in \eqref{BVcandidate} and $I^0$ defined in \eqref{IoBV}.
\end{prop}

\begin{proof}
For every $u\in BV(\omega;\mathcal{M})$ we can write
\begin{equation}
\label{BMBV5.1}
    I^0(u,\omega)=I^0(u, \omega\setminus S_u)+ I^0(u, S_u),
\end{equation}
where $S_u$ denotes the jump set of $u$. It follows that in order to prove our claim we can separately estimate the two terms on the right hand side of \eqref{BMBV5.1}.
First, by Proposition \cite[Proposition 3.3]{LTZ1} we have for every $A\in\mathcal{A}(\omega)$
\[
I^0(u,A)\le 
\left\{\begin{array}{ll}
\displaystyle \int_A Tf_{\rm hom}^0(u,\nabla_\alpha u)\,dx & \text{if $u\in W^{1,1}(\omega;\mathcal M)$}\\
\\
+\infty & \text{otherwise in $L^1(\omega;\mathbb R^3)$.}
\end{array}\right.
\]
By means of a relaxation argument (see \cite[Thm.\,7.1]{BMbv} which holds by Proposition \ref{characterization}) and exploiting the coercivity of $f$ (see H2), a standard diagonalization argument, and the lower semicontinuity of $I^0$ with respect to $L^1(\Omega,\R^3)$ strong convergence we get
\begin{equation}\label{limsupBV}
\begin{aligned}
    I^0(u, A) \le \, &\int_{A} T f^0_{\rm hom} (u, \nabla_\alpha u) \,dx_\alpha\\ &+\int_{S_u\cap A} H (u^+, u^-, \nu_u) \, d \mathcal{H}^{1}+\int_{A} T f_{\rm hom}^{0,\infty} \left (\tilde{u}, \frac{d D^c_\alpha u}{d |D^c_\alpha u|} \right ) d |D^c_\alpha u|
\end{aligned}
\end{equation}
for some function $H:\mathcal{M}\times \mathcal{M}\times\mathbb{S}^1\to [0+\infty).$ By outer regularity, in \eqref{limsupBV} we can choose $A=\omega\setminus S_u$, which leads to
\begin{align*}
    I^0(u, \omega\setminus S_u)\le &\int_{A} T f^0_{\rm hom} (u, \nabla_\alpha u) \,dx_\alpha +\int_{A} T f_{\rm hom}^{0,\infty} \left (\tilde{u}, \frac{d D^c_\alpha u}{d |D^c_\alpha u|} \right ) d |D^c_\alpha u|.
\end{align*}
Now we estimate $I^0(u, S_u)$. By Proposition \ref{BMBVprop_4.1} we have
\begin{equation*}
    I^0(u,S)=\int_S K (u^+, u^-, \nu_u) \, d \mathcal{H}^{1},
\end{equation*}
for every $u \in BV(\omega; T)$ and every Borel subset $S$ of $\omega \cap S_u,$ so to conclude it is enough to prove that
\begin{equation*}
    K(a, b, \nu) \le \, \theta(a, b, \nu)
\end{equation*}
for every $(a, b, \nu) \in \mathcal{M} \times \mathcal{M} \times \mathbb{S}^1.$
\color{black}
Since the proof is similar to Proposition \ref{BMBV3.2_3.3} we refer to it for the notation. 
Let $\nu = (\nu_1,\nu_2)$ be an orthonormal basis of $\mathbb{R}^2$. Since $K$ and $\theta$ are continuous functions in the $\mathbb{S}^{1}$ variable, we can assume that $\nu$ is a rational basis, i.e. for all $i \in \{1,2\}$, there exists $\gamma_i \in \mathbb{R} \setminus \{0\}$ such that $v_i := \gamma_i \nu_i \in \mathbb{Z}^2$. The general case, then, follows by a density argument. Given an arbitrary $0 < \eta < 1$, by Proposition \ref{BMBV3.2_3.3} and \eqref{BMBV3.29}, there exists $h_0 > 0$, $u_0 \in \mathcal B_{h_0}(a, b, \nu)$ and $\gamma_{h_0} \in\mathcal G(a, b)$ such that
	\[
	u_0(x) = \gamma_{h_0}\left( \frac{x_\alpha \cdot \nu_1}{h_0} \right) \quad \text{and} \quad \int_{Q_{\nu,1}} f^\infty\left( \frac{x_\alpha}{h_0},x_3, \nabla_{h_0} u_0 \right) dx \leq \theta(a, b, \nu_1) + \eta.
	\]	
	Given $\lambda \in \mathbb{Z}$, let $(h_n)_n$ be a vanishing positive sequence and set
	\[
	x_n^{(\lambda)} := h_n \lambda v_2, \quad Q_{\nu, n}^{(\lambda)} := x_n^{(\lambda)} + \left(\frac{h_n}{h_0}\right) Q_{\nu,1}.
	\]
	Next, we define
\begin{align*}
	\Lambda_n := \Bigg\{ \lambda \in \mathbb{Z} : Q_{\nu, n}^{(\lambda)} \subset Q_{\nu,1} &\text{ and } \frac{x_n^{(\lambda)}}{h_n} \in  l \left( \frac{1}{h_0} +  \gamma_2 \right) \nu_2 + P \text{ for some } l \in \mathbb{Z} \Bigg\},
\end{align*}
	with
	\[
	P := \left\{  \alpha v_2 : \alpha \in [-1/2, 1/2) \right\}.
	\]
	Now we consider the sequence of functions
	\[
	u_n(x_\alpha, x_3) := 
	\begin{cases}
		u_0\left( \frac{h_0}{h_n}(x_\alpha - x_n^{(\lambda)})\cdot\nu_1 \right) & \text{if } x \in Q_{\nu, n}^{(\lambda)} \text{ for some } \lambda \in \Lambda_n, \\
		\gamma_{h_0}\left( \frac{x_\alpha \cdot \nu_1}{h_0} \right) & \text{otherwise}.
	\end{cases}
	\]
	By construction, $u_n \in W^{1,1}(Q_{\nu,1}; \mathcal{M})$, $(\nabla_{h_n} u_n)_n$ is bounded in $L^1(Q_{\nu,1}; \mathbb{R}^{3 \times 2})$, and $u_n \to u^{a,b}_{\nu_1}$ in $L^1(Q_{\nu,1}; \mathbb{R}^3)$ as $n \to +\infty$ with
	\[
	u^{a,b}_{\nu_1}(x) :=
	\begin{cases}
		a & \text{if } x_\alpha \cdot \nu_1 \geq 0, \\
		b & \text{if } x_\alpha \cdot \nu_1 < 0
	\end{cases}.
	\]
	By a similar argument as in Step 1 of the proof of \cite[Proposition 2.2]{BDFV}, we get that
	\begin{align}
		\label{BMBV5.3}
		\limsup_{n \to +\infty} \int_{Q_{\nu,1}} f^\infty\left( \frac{x_\alpha}{h_n},x_3, \nabla_{h_n} u_n \right) dx& \leq \int_{Q_{\nu,1}} f^\infty\left( \frac{x_\alpha}{h_0},x_3, \nabla_{h_0} u_0 \right) dx\nonumber\\
		&\leq \theta(a, b, \nu_1) + \eta.
	\end{align}
	For $\rho > 0$, let $A_\rho := Q_{\nu,1} \cap \{ |x_\alpha \cdot \nu_1| < \rho \}_{,1}$ and set $\Pi_{\nu_1} := \{ x_\alpha \in \mathbb{R}^2 : x_\alpha \cdot \nu_1 = 0 \}$. By construction, the sequence $(u_n)_n$ is admissible for $I^0(u^{a,b}_{\nu_1}, A_\eta)$, so that
	\begin{align}
		\label{BMBV5.4}
		I^0(u^{a,b}_{\nu_1}, A_\eta \cap \Pi_{\nu_1}) &\leq I^0(u^{a,b}_{\nu_1}, A_\eta) \leq \liminf_{n \to +\infty} \int_{A_{\eta,1}} f\left( \frac{x_\alpha}{h_n},x_3, \nabla_{h_n} u_n \right) dx \nonumber\\
		&\leq \beta \mathcal{L}^2(A_\eta) + \liminf_{n \to +\infty} \int_{A_{h_n,1}} f\left(  \frac{x_\alpha}{h_n},x_3, \nabla_{h_n} u_n \right) dx \nonumber\\
		&\leq \liminf_{n \to +\infty} \int_{A_{h_n,1}} f\left(  \frac{x_\alpha}{h_n},x_3, \nabla_{h_n} u_n\right) dx + \beta \eta.
	\end{align}
	Here we used assumption (H2) and the fact that $\nabla_{h_n} u_n = 0$ outside $A_{h_n,1}$. On the other hand, by \eqref{BMBV_prop4.1} we have
	\begin{equation}
		\label{BMBV5.5}
		I^0(u^{a,b}_{\nu_1}, A_\eta \cap \Pi_{\nu_1}) = \int_{A_\eta \cap \Pi_{\nu_1}} K(a, b, \nu_1) \, d\mathcal{H}^{1} = K(a, b, \nu_1).
	\end{equation}
	Using (H4), the boundedness of $(\nabla_{h_n} u_n)_n$ in $L^1(Q_{\nu,1}; \mathbb{R}^{3 \times 2})$, the fact that $f^\infty(\cdot, 0) \equiv 0$, and Hölder's inequality, we derive
	\begin{align}
		\label{BMBV5.6}
		&\left| \int_{A_{h_n,1}} f\left(  \frac{x_\alpha}{h_n},x_3, \nabla_{h_n} u_n \right) dx - \int_{Q_{\nu,1}} f^\infty\left(  \frac{x_\alpha}{h_n},x_3, \nabla_{h_n} u_n \right) dx \right|\nonumber\\
		&\leq C \int_{A_{h_n,1}} \left( 1 + |\nabla_{h_n} u_n|^{1-q} \right) dx \leq C \left( h_n + h_n^q \|\nabla_{h_n} u_n\|^{1-q}_{L^1(Q_{\nu,1}; \mathbb{R}^{3 \times 2})} \right) \to 0.
	\end{align}
	From \eqref{BMBV5.3}, \eqref{BMBV5.4}, \eqref{BMBV5.5}, \eqref{BMBV5.6} we get
	\begin{equation*}
		K(a, b, \nu_1)\le \theta(a, b, \nu_1)+ (1+\beta)\eta,
	\end{equation*}
	and from the arbitrariness of $\eta$ we get the desired inequality.
     By exploiting Proposition \ref{BMbvProp3.4}, Proposition \ref{BMBV_prop4.1}, and Lemma \ref{LocalizationBV} we can argue as in \cite[Corollary 5.1]{BMbv}, which, in turn rely on the approximation arguments in \cite[Lemma 4.7 and Proposition 4.8]{AMT}, and together with \eqref{BMBV5.1} we get

    \[
\begin{aligned}
    I^0(u) \le \, &\int_{\omega} T f^0_{\rm hom} (u, \nabla_\alpha u) \,dx_\alpha\\ &+\int_{S_u\cap\omega} \theta (u^+, u^-, \nu_u) \, d \mathcal{H}^{1}+\int_{\omega} T f_{\rm hom}^{0,\infty} \left (\tilde{u}, \frac{d D^c_\alpha u}{d |D^c_\alpha u|} \right ) d |D^c_\alpha u|=I(u),
\end{aligned}
    \]
    where $I$ is defined in \eqref{BVcandidate}, which yields the conclusion.
\end{proof}

\noindent
This concludes the proof of the upper bound. We now prove the lower bound using the blow up method.

\begin{prop}\label{liminfbv}
    For every $u\in BV(\omega,\mathcal{M})$ it holds
    \begin{equation*}
        I^0(u)\ge I(u),
    \end{equation*}
    with $I$ defined in \eqref{BVcandidate} and $I^0$ defined in \eqref{IoBV}.
\end{prop}

\begin{proof}

Consider $u \in BV(\omega; \mathcal{M})$ and let $(u_n)_n \subset W^{1,1}(\Omega; \mathcal{M})$ be such that
\[
I^0(u, \omega) = \lim_{n \rightarrow + \infty} \int_{\Omega} f \left ( \frac{x_\alpha}{h_n},x_3, \nabla_{h_n} u_n\right ) \, dx.
\]
Define the sequence of nonnegative Radon measures
\[
\mu_n :=  \left(\int_{(-\frac{1}{2}, \frac{1}{2})}f \left (\frac{(\cdot)_\alpha}{h_n}, x_3, \nabla_{h_n} u_n((\cdot)_\alpha,x_3) \right )dx_3\right) \mathcal{L}^{2} \mres \omega.
\]
Possibly extracting a subsequence, we can infer the existence of a nonnegative Radon measure $\mu \in \mathcal{M}(\omega)$ such that $\mu_n \stackrel{^*}{\rightharpoonup} \mu$ in $\mathcal{M}(\omega).$ Exploiting the Besicovitch Differentiation Theorem, we are able to split $\mu$ into the sum of three nonnegative measure which are mutually singular, namely
\[
\mu = \mu^a + \mu^j + \mu^c
\]
where 
\[
\mu^a \ll \mathcal{L}^2, \qquad \mu^j \ll \mathcal{H}^{1} \mres S_u \qquad \mu^c \ll |D_\alpha^c u|.
\]
As long as $\mu(\omega) \le \, I^0(u, \omega),$ it is sufficient to check that
\begin{equation}
\label{BM(6.1)}
\frac{d \mu}{d \mathcal{L}^2}(x_0) \ge \, T f_{\rm hom}^0(u(x_0), \nabla_\alpha u(x_0)) \qquad \textnormal{for $\mathcal{L}^2-$a.e. $x_0 \in \omega,$}
\end{equation}

\begin{equation}
\label{BM(6.2)}
\frac{d \mu}{d |D_\alpha^c u|}(x_0) \ge \, T f_{\rm hom}^{0,\infty} \left (\tilde{u}(x_0), \frac{d D_\alpha^c u}{d |D_\alpha^c u|} (x_0) \right ) \qquad \textnormal{for $|D_\alpha^c u|-$a.e. $x_0 \in \omega,$}
\end{equation}
and

\begin{equation}
\label{BM(6.3)}
\frac{d \mu}{d \mathcal{H}^{1} \mres S_u} (x_0) \ge \, \theta(u^+(x_0), u^-(x_0), \nu_u(x_0)) \qquad \textnormal{for $\mathcal{H}^{1}-$a.e. $x_0 \in S_u.$}
\end{equation}

\vspace{2mm}

$\bullet \textsc{proof of \eqref{BM(6.1)}}$
Before proving the result we recall that all the operations of sum and difference between the functions $u_n, v_n, w_k, w_{n,k}$ and $u$ must be inteded in the sense of Remark \ref{abuso}.\\

\noindent
{\sc step 1.} Consider $x_0 \in \omega$ to be a Lebesgue point of $u$ and $\nabla_\alpha u,$ a point of approximate differentiability of $u,$ in a way that $u(x_0) \in \mathcal{M}, \nabla_\alpha u(x_0) \in [T_{u(x_0)}(\mathcal{M})]^2$ and such that the Radon-Nykod\'ym derivative of $\mu$ with respect to the Lebesgue measure $\mathcal{L}^2$ exists and it is finite. We observe that $\mathcal{L}^2-$almost every points $x_0 \in \omega$ satisfy these properties. We set 
\[
s_0 := u(x_0) \qquad \xi_0 := \nabla_\alpha u(x_0).
\]
We may assume that, up to a subsequence, there exists a nonnegative Radon measure $\lambda \in \mathcal{M}(\Omega)$ such that
\[
\left(\int_{-\frac{1}{2}}^{\frac{1}{2}}(1 + |\nabla_{h_n} u_n|)dx_3\right) \mathcal{L}^2 \mres\omega \stackrel{*}{\rightharpoonup} \lambda \qquad \textnormal{in $\mathcal{M}(\omega).$}
\]
Let us consider a sequence $(\rho_k)_k \rightarrow 0^+$ such that $Q'(x_0, 2 \rho_k) \subset \Omega$ and $\mu(\partial Q'(x_0, \rho_k)) = \lambda (\partial Q(x_0, \rho_k)) = 0$ for each $k \in \mathbb{N}.$ Since $\mu$ is the limit of $\mu_n$, we have that
\begin{equation}
\label{(5.9)}
\mu(Q'(x_0, \rho_k)) = \lim_{n \rightarrow + \infty} \int_{Q'(x_0, \rho_k)_1} f\left (\frac{x_\alpha}{h_n}, x_3, \nabla_{h_n} u_n \right ) \, dx.
\end{equation}
Since
\[
\tau_n := h_n \left [ \frac{x_0}{h_n}\right ] \in h_n \mathbb{Z}^2 \rightarrow x_0,
\]
given $r \in (0, 1/2)$ we have that $Q'(\tau_n, \rho_k) \subset Q'(x_0, r \rho_k)$ for sufficiently large values of $k;$ therefore we can define, for $x \in Q'(0, \rho_k)_1$
\begin{equation}
\label{(5.9)bis}
v_n(x) := u_n(x_\alpha + \tau_n, x_3).
\end{equation}
The continuity of the translation in $L^1$ allows us to obtain the following estimate, holding for $n \rightarrow + \infty$
\begin{eqnarray}
\int_{Q'(0, \rho_k)_1} |v_n(x) - u(x_\alpha + x_0)| \, dx &=& \int_{Q'(\tau_n, \rho_k)_1} |u_n(x) - u(x_\alpha + x_0 - \tau_n)| \, dx \label{(5.10)}\\
&\le & \int_{Q'(x_0, r \rho_k)_1} |u_n(x) - u(x_\alpha + x_0 - \tau_n)| \, dx \rightarrow 0. \nonumber
\end{eqnarray}
A change of variable in \eqref{(5.9)} together with the periodicity condition (H1) of $f(\cdot,x_3, \xi),$ the growth condition (H2), and \eqref{(5.9)bis} lead to
\begin{eqnarray}
\mu(Q'(x_0, \rho_k)) &=& \lim_{n \rightarrow + \infty} \int_{Q'(x_0 - \tau_n, \rho_k)_{,1}} f \left (\frac{x_\alpha + \tau_n}{h_n},x_3, \nabla_{h_n} u_n(x_\alpha + \tau_n, x_3) \right ) \, dx \nonumber\\
&=& \lim_{n \rightarrow + \infty} \int_{Q'(x_0 - \tau_n, \rho_k)_{,1}} f \left (\frac{x_\alpha}{h_n}, x_3, \nabla_{h_n} v_n\right ) \, dx \nonumber\\
&\ge& \limsup_{n \rightarrow + \infty} \int_{Q'(0, \rho_k)_{,1}} f \left (\frac{x_\alpha}{h_n}, x_3, \nabla_{h_n} v_n\right ) \, dx \nonumber\\
&& - \beta \limsup_{n \rightarrow + \infty} \int_{Q_1(\tau_n, \rho_k)_{,1} \setminus Q'(x_0, \rho_k)_{,1}} (1 + |\nabla_{h_n} u_n|) \, dx. \label{(5.11)}
\end{eqnarray}
The last term in \eqref{(5.11)} vanishes since, by the choice of $\rho_k$
\begin{eqnarray*}
&& \limsup_{n \rightarrow + \infty} \int_{Q'(\tau_n, \rho_k)_{,1} \setminus Q'(x_0, \rho_k)_{,1}} (1 + |\nabla_{h_n} u_n|) \, dx \\
&\le & \, \limsup_{r \rightarrow 1^+} \limsup_{n \rightarrow + \infty} \int_{Q'(x_0, r \rho_k)_{,1} \setminus Q'(x_0, \rho_k)_{,1}} (1 + |\nabla_{h_n} u_n|) \, dx\\
&\le & \limsup_{r \rightarrow 1^+} \lambda \left (\overline{Q'(x_0, r \rho_k)_{,1} \setminus Q'(x_0, \rho_k)_{,1}} \right )\\
&\le & \lambda (\partial Q'(x_0, \rho_k)_{,1}) = 0.
\end{eqnarray*}
Summing up, this entails
\[
\mu(Q'(x_0, \rho_k) \ge \,  \limsup_{n \rightarrow + \infty} \int_{Q'(0, \rho_k)_{,1}} f \left (\frac{x_\alpha}{h_n}, x_3, \nabla_{h_n} v_n\right ) \, dx,
\]
where $(v_n)_n \subset W^{1,1}(Q'(0, \rho_k)_{,1}; \mathcal{M})$ is such that $v_n \rightarrow u(x_0 + (\cdot)_\alpha)$ in $L^1(Q'(0, \rho_k)_{,1}; \mathbb{R}^3)$ thanks to \eqref{(5.10)}.
\\
At this point, let us consider, for every $n$ a sequence $(v_{n, j})_j \subset \mathcal{C}^{\infty}(\overline{Q'(0, \rho_k)_{,1}}; \mathbb{R}^3)$ (not necessarily taking values in $\mathcal{M}$!) such that $v_{n,j} \rightarrow v_n$ in $W^{1,1}(Q'(0, \rho_k)_{,1}; \mathbb{R}^3),$ $v_{n,j} \rightarrow v_n$ and $\nabla v_{n,j} \rightarrow \nabla v_n$ a.e. in $Q'(0, \rho_k)_{,1}$ as $j \rightarrow + \infty$ and consider the function $g$ introduced in \eqref{(2.6)}. It is possible to prove that
\begin{eqnarray*}
&& \lim_{j \rightarrow + \infty} \int_{Q'(0, \rho_k)_{,1}} g \left (\frac{x_\alpha}{h_n}, x_3, v_{n,j}, \nabla_{h_n} v_{n,j} \right ) \, dx \\
&=& \int_{Q'(0, \rho_k)_{,1}} g \left (\frac{x}{h_n}, x_3, v_{n}, \nabla_{h_n} v_{n} \right ) \, dx \\
&=& \int_{Q'(0, \rho_k)_{,1}} f \left (\frac{x_\alpha}{h_n}, x_3, \nabla_{h_n} v_{n} \right ) \, dx,
\end{eqnarray*}
therefore, by possibly passing to a diagonal sequence $\bar{v}_n := v_{n, j_n},$ such that $\bar{v}_n \rightarrow u(x_0 + (\cdot)_\alpha)$ in $L^1(Q'(0, \rho_k)_{,1}; \mathbb{R}^3),$ we may deduce
\[
\mu(Q'(x_0, \rho_k) \ge \,  \limsup_{n \rightarrow + \infty} \int_{Q'(0, \rho_k)_1} g \left (\frac{x_\alpha}{h_n}, x_3, \bar{v}_n, \nabla_{h_n} \bar{v}_n \right ) \, dx.
\]
If we perform a change of variable in the previous inequality and set
\[
w_{n,k}(x) := \frac{\bar{v}(\rho_k x_\alpha, x_3) - s_0}{\rho_k},
\]
we may deduce
\begin{eqnarray}
\frac{d \mu}{d \mathcal{L}^2}(x_0) &\ge &\, \limsup_{k \rightarrow + \infty} \limsup_{n \rightarrow + \infty} \int_Q g \left (\frac{\rho_k x_\alpha}{h_n}, x_3, \bar{v}_n(\rho_k x_\alpha, x_3), \nabla_{\frac{\rho_k}{h_n}} \bar{v}_n(\rho_k x_\alpha, x_3) \right ) \, dx \nonumber\\ 
&= &\, \limsup_{k \rightarrow + \infty} \limsup_{n \rightarrow + \infty} \int_Q g \left (\frac{\rho_k x_\alpha}{h_n}, x_3, s_0 + \rho_k w_{n, k}, \nabla_{\frac{\rho_k}{h_n}} w_{n,k} \right ) \, dx. \label{(5.13)}
\end{eqnarray}
Since $x_0$ is a point of approximate differentibility of $u$ and $\bar{v}_n \rightarrow u(x_0 + (\cdot)_\alpha)$ in $L^1(Q'(0, \rho_k)_{,1}; \mathbb{R}^3),$ we have
\begin{equation}
\label{(5.14)}
\lim_{k \rightarrow + \infty} \lim_{n \rightarrow + \infty} \int_Q [w_{n,k}(x) - \xi_0 x_\alpha] \, dx = \lim_{k \rightarrow + \infty} \int_{Q'(x_0, \rho_k)} \frac{|u(x_\alpha) - s_0 - \xi_0 (x_\alpha - x_0)|}{\rho_k^{3}} \, dx_\alpha = 0.
\end{equation}
At this point, \eqref{(5.13)} and \eqref{(5.14)} allow us to take a diagonal sequence $h_{n_k} < \rho^2_k$ in such a way that $w_k := w_{n_k,k} \rightarrow w_0$ in $L^1(Q; \mathbb{R}^d)$ with $w_0(x) := \xi_0 x_\alpha$ and
\begin{equation*}
\label{(5.15)}
\frac{d \mu}{d \mathcal{L}^2} (x_0) \ge \, \limsup_{k \rightarrow + \infty} \int_{Q}g \left (\frac{x}{\delta_k},x_3, s_0 + \rho_k w_k, \nabla_{\delta_k} w_k \right ) \, dx,
\end{equation*}
where $\delta_k := \frac{h_{n_k}}{\rho_k} \rightarrow 0.$ 
\\
\\ 
{\sc step 2.} This second step can be performed following the lines of \cite[Lemma 5.2]{BM}, which, in turn, rely on \cite[Theorem 2.1]{FM}. In fact we observe that the rescaled gradient can be seen as the gradient of a composed function. Moreover, the dependence on $x_3$ and the 'homogenization procedure' in the first two variables and the dimension reduction with respect to the third one do not influence the strategy and the estimates, which rely only on the growth of $f$.

Thus we can show the existence of a sequence $(\overline{w}_k)_k \subset W^{1, \infty}(Q; \mathbb{R}^3)$ such that $\overline{w}_k \rightarrow w_0$ in $L^{\infty}(Q; \mathbb{R}^3),$ $\overline{w}_k$ is uniformly bounded in $W^{1,1}(Q; \mathbb{R}^3)$ and $\nabla_{\delta_k}\overline{w}_k$ is bounded in $L^1(Q;\mathbb R^{3\times 3})$.

\begin{equation}
\label{(5.16)}
\frac{d \mu}{d \mathcal{L}^2}(x_0) \ge \, \limsup_{k \rightarrow + \infty} \int_Q g \left ( \frac{x_\alpha}{\delta_k}, x_3, s_0 + \rho_k \overline{w}_k, \nabla_{\delta_k} \overline{w}_k\right )dx.
\end{equation}

{\sc step 3:} Thanks to the fact that $(\overline{w}_k)_k\subset{L^{\infty}(Q; \mathbb{R}^3)}$ and  $(\nabla_{\delta_k} \overline{w}_k)_k\subset{L^{\infty}(Q; \mathbb{R}^{3 \times 3})}$ are uniformly bounded in $L^1$, \eqref{(2.8)} gives
\[
\lim_{k \rightarrow + \infty} \int_Q \left | g \left ( \frac{x_\alpha}{\delta_k}, x_3, s_0 + \rho_k \overline{w}_k, \nabla_{\delta_k} \overline{w}_k\right ) - g \left ( \frac{x_\alpha}{\delta_k}, x_3, s_0, \nabla_{\delta_k} \overline{w}_k\right )\right | \, dx = 0,
\]
which in turn entails, together with \eqref{(5.16)}
\[
\frac{d \mu}{d \mathcal{L}^2}(x_0) \ge \, \limsup_{k \rightarrow + \infty} \int_Q g \left ( \frac{x_\alpha}{\delta_k}, x_3, s_0, \nabla_{\delta_k} \overline{w}_k\right )dx.
\]
Now we claim that
\begin{equation}
\label{caratt_limsup_puno}
    \limsup_{k \rightarrow + \infty} \int_Q g \left ( \frac{x_\alpha}{\delta_k}, x_3, s_0, \nabla_{\delta_k} \overline{w}_k\right )dx=g_{\rm hom}^0(s_0, \xi_0), 
\end{equation}
with the latter defined in \eqref{f0hom}. 
This concludes the proof since by \eqref{(2.3)} we have
\[
\frac{d \mu}{d \mathcal{L}^2}(x_0)\ge \, g_{\rm hom}^0(s_0, \xi_0) = T f_{\rm hom}^0(s_0, \xi_0).
\]
The proof of \eqref{caratt_limsup_puno} is the following. 
Consider the family of integral functionals 
\begin{equation*}
   (w,A)\in W^{1,1}(Q';\R^3)\times\mathcal{A}(Q')\mapsto G^{\delta_k}(w,A):=\int_{A}\int_{-\frac{1}{2}}^{\frac{1}{2}} g \left ( \frac{x_\alpha}{\delta_k}, x_3, s_0, \nabla_{\delta_k} w\right )dx_3dx_\alpha,
\end{equation*}
for some vanishing sequence $(\delta_k)_k$. Let $G:BV(Q',\R^3)\times \mathcal{A}(Q')\to \R$ be 
\begin{equation*}
    G(w, A):=\inf\left\{\liminf_{k\to+\infty}G^{\delta_k}(w_k,A) : w_k \to w\,\text{ in $L^1(A;\mathcal M)$}\right\}.
\end{equation*}
It follows that $G$ is $L^1$ lower semicontinuous. Moreover, since $g$ satisfies hypothesis (H1) to (H4), then by the same argument as Lemma \ref{LocalizationBV} we get that $G$ is the restriction on $\mathcal{A}(Q')$ of a positive Radon measure. Furthermore, by Lemma \ref{extension}-(ii) it easy to see that
\begin{equation*}
    \alpha'|D_\alpha u|(A)\le G(w,A)\le \beta'(\mathcal{L}^2(A)+|D_\alpha u|(A)),
\end{equation*}
for every $(w,A) \in BV(Q';\R^3)\times\mathcal{A}(Q')$. Standard arguments ensure that $G(w+c,A)=G(w,A)$ for any $c\in\R^3$ and $(w,A) \in BV(Q';\R^3)\times\mathcal{A}(Q')$. 


By Lemma \ref{extension}-(i) and classical arguments in homogenization theory, it is possible to deduce that
 \[  
     G(\tau_yu,y+A)=G(u,A),
     \]
    for any $A\in\mathcal{A}(Q')$, $y\in \R^2$ such that $y+A\subset Q'$, $u\in BV(Q'; \mathcal{M})$, and with $\tau_yw(x):=w(x-y)$.
    Combining the last to identities we get that
    \begin{equation*}
        G(\tau_yw+c,y+A)=G(w,A).
    \end{equation*}

    Now we apply \cite[Theorem 3.12, and Remark 3.8 (1) and (2)]{BFM} and we obtain that
    \begin{align*}
        G(w,A)=\int_A \varphi^b&(s_0,\nabla_\alpha w)dx_\alpha+\int_{S_w\cap A} \varphi^j(s_0,(w^--w^+)\otimes \nu_w)d\mathcal{H}^1(x_\alpha)\\
        &+ \int_A \varphi^{b, \infty}\left(s_0, \frac{dD_\alpha^c w}{d|D_\alpha^c w|}\right)d|D_\alpha^c w|,
    \end{align*}
    for every $(w,A) \in BV(Q';\R^3)\times\mathcal{A}(Q')$. In what follow, we prove that, for every ${\xi_\alpha} \in \R^{3\times 2}$,
    \begin{align*}
        \varphi^b(s_0, {\xi_\alpha})&= \lim_{t \rightarrow + \infty} \inf_{\varphi} \left \{ \media_{(tQ')_{,1}} g(y, s_0, {\xi_\alpha} + \nabla_\alpha \varphi(y)|\nabla_3\varphi(y)) \, dy: \right. \\
& \left. \varphi \in W^{1, \infty}((tQ')_{,1}; \mathbb{R}^3),\, \varphi(x_{\alpha}, x_3) = 0\,\,\, \textnormal{for every $(x_{\alpha}, x_3) \in \partial (tQ') \times \left(-\frac{1}{2},\frac{1}{2}\right)$} \right \}\\
&=:g^0_{\rm hom}(s_0, {\xi_\alpha}), 
    \end{align*}
 in \eqref{f0hom}.
 The proof of this identity follows the same argument as \cite[pages 1390-1391]{BFF}, and reasoning as in \cite[Proposition 2.4]{BM84} we can substitute the space $W^{1,p}$ with $W^{1,\infty}$.\\

\noindent
$\bullet$ \textsc{proof of \eqref{BM(6.2)}.}
For this proof we follow the strategy of \cite[section 5.1]{AEL07}. Let $G:BV(\omega, \R^3)\to \R$ be defined as 
 \begin{equation*}
     G(u):= \int_{\omega}g^0_{hom}(u,\nabla_\alpha u)dx_\alpha.
 \end{equation*}
By \eqref{BM(6.1)} we have that for every $u\in BV(\omega,\mathcal{M})$ it holds
 \begin{equation*}
     I^0(u)\ge G(u).
 \end{equation*}
So, in particular, it holds
\begin{equation*}
     I^0(u)\ge \overline{G}(u), \quad u\in BV(\omega,\mathcal{M}),
 \end{equation*}
with
\begin{equation*}
    \overline{G}(u):=\inf\left\{\liminf_{n\to \infty}G(u_n): u_n \rightharpoonup^* u \text{ in }BV(\omega,\R^3)\right\}.
\end{equation*}
For $\e>0$, consider now the function $\tilde g(s,\xi_\alpha):\R^3\times \R^{3\times 2}\to \R$ defined as
\begin{equation*}
    \tilde g_\e(s,\xi_\alpha):=g^0_{hom}(x,\xi_\alpha) +\e|\xi_\alpha|.
\end{equation*}
Fixed $u\in BV(\omega,\mathcal{M})$ and a sequence $(u_n)_n\subset BV(\omega,\mathcal{M})$ such that $ u_n \rightharpoonup^* u$ in $BV(\omega,\mathcal{M})$, then in particular $u_n \rightharpoonup^* u$ in $BV(\omega,\R^{3})$ and
\begin{align*}
    \liminf_{n\to \infty} G(u_n)&\ge \liminf_{n\to \infty} \int_{\omega}\tilde{g}_\e(u_n,\nabla_\alpha u_n)dx_\alpha-\e\sup_{n}||\nabla_\alpha u_n||_{L^1(\omega,\R^{3\times 2})}\\
    & \ge \int_\omega \tilde{g}_\e(u,\nabla_\alpha u)dx_\alpha + \int_\omega \tilde{g}^{\infty}_\e\left(u,D^c_\alpha \frac{d D_\alpha^c u}{d |D_\alpha^c u|}\right)d |D_\alpha^c u| -\e\sup_{n}||\nabla_\alpha u_n||_{L^1(\omega,\R^{3\times 2})},
\end{align*}
where in the last inequality we used \cite[Theorem 2.16]{FM93}. Passing to the limit for $\e \to 0$, then considering that $\tilde{g}^{\infty}_\e(s,\xi_\alpha)=g^{0,\infty}_{hom}(x,\xi_\alpha) +\e$ we get
\begin{align*}
    I^0(u)&\ge \int_\omega g_{hom}^0(u,\nabla_\alpha u)dx_\alpha + \int_\omega g_{hom}^{0,\infty}\left(u,\frac{d D_\alpha^c u}{d |D_\alpha^c u|}\right)d |D_\alpha^c u|\\
    & = \int_\omega Tf_{hom}^0(u,\nabla_\alpha u)dx_\alpha + \int_\omega Tf_{hom}^{0,\infty}\left(u,\frac{d D_\alpha^c u}{d |D_\alpha^c u|}\right)d |D_\alpha^c u|
\end{align*}
 wich proves \eqref{BM(6.2)}, passing to densities via Besicovitch's differentiation.
\vspace{2mm}

$\bullet$ \textsc{proof of \eqref{BM(6.3)}.} In this step we use a strategy similar to the one employed in \cite[Lemma 6.1]{BMbv}. In particular, we will still rely on the blow up method together with the projection argument in \cite[Proposition 2.1]{BMbv} which we already used in the localization (see Lemma \ref{LocalizationBV}).
\\
\\
{\sc step 1.} Let $x_0 \in S_u$ be such that
\begin{equation}
\label{BM(6.25)}
\lim_{\rho \rightarrow 0^+} \media_{Q^{\pm}_{\nu_u(x_0)}(x_0, \rho)} |u(x_\alpha) - u^{\pm}(x_0)| \, dx_\alpha = 0,
\end{equation}
where $u^{\pm}(x_0) \in \mathcal{M},$
\begin{equation}
\label{BM(6.26)}
\lim_{\rho \rightarrow 0^+} \frac{\mathcal{H}^{1}(S_u \cap Q_{\nu_u(x_0)}(x_0, \rho))}{\rho} = 1,
\end{equation}
and such that the Radon-Nikod\'ym derivative of $\mu$ with respect to $\mathcal{H}^{1} \mres S_u$ exists and it is finite. By \cite[Lemma 2.1, Theorem 3.78, Theorem 2.83 (i)]{AFP} with cubes instead of balls, it turns out that $\mathcal{H}^{1}$-a.e. $x_0 \in S_u$ satisfy these properties. Now let us set $s_0^{\pm} := u^{\pm}(x_0)$ and $ \nu_0 := \nu_u(x_0).$ Given a sequence $(u_n)_n\subset W^{1,1}(\Omega,\R^3)$ such that $\liminf_{n\to \infty} I^{h_n}(u_n)\le c$, then, possibly extracting a further subsequence, we may assume that 
\[
\left(\int_{-\frac{1}{2}}^{\frac{1}{2}}(1 + |\nabla_{h_n} u_n|)dx_3 \right)\mathcal{L}^2 \mres \omega \stackrel{*}{\rightharpoonup} \lambda \qquad \textnormal{in $\mathcal{M}(\omega)$},
\]
for some nonnegative Radon measure $\lambda \in \mathcal{M}(\Omega).$ At this point, let us consider a sequence $\rho_k \rightarrow 0^+$ satisfying 
\[
\mu(\partial Q'_{\nu_0}(x_0, \rho_k)) = \lambda (\partial Q'_{\nu_0}(x_0, \rho_k)) = 0,
\]
for each $k \in \mathbb{N}.$ By \eqref{BM(6.26)}, we have
\begin{eqnarray*}
&& \frac{d \mu}{d \mathcal{H}^{1} \mres S_u}(x_0) =\lim_{k \rightarrow + \infty} \frac{\mu(Q'_{\nu_0}(x_0, \rho_k))}{\mathcal{H}^1(S_u \cap Q'_{\nu_0}(x_0, \rho_k))} \\ 
&=& \lim_{k \rightarrow + \infty}  \frac{\mu(Q'_{\nu_0}(x_0, \rho_k))}{\rho_k} \\
&=& \lim_{k \rightarrow + \infty} \lim_{n \rightarrow + \infty} \frac{1}{\rho_k} \int_{Q'_{\nu_0}(x_0, \rho_k)_{,1}} f \left (\frac{x_\alpha}{h_n},x_3, \nabla_{h_n} u_n \right ) \, dx.
\end{eqnarray*} 
By exploiting \cite[Theorem 2.1]{BMbv}, it is possible to assume, without loss of generality, that $u_n \in \mathcal{D}(\Omega; \mathcal{M})$ (see Theorem \ref{label}) for each $n \in \mathbb{N}.$ Therefore, arguing as in the proof of \eqref{BM(6.1)} (with $Q'_{\nu_0}(x_0, \rho_k)$ instead of $Q'(x_0, \rho_k)$) we get a sequence $(v_n)_n \subset \mathcal{D}((Q'_{\nu_0}(0, \rho_k))_1; \mathcal{M})$ such that $v_n \rightarrow u(x_0 + \cdot)$ in $L^1((Q'_{\nu_0}(0, \rho_k))_1; \mathbb{R}^3)$ as $n \rightarrow + \infty$ and
\[
\frac{d \mu}{d \mathcal{H}^{1} \mres S_u}(x_0) \ge \, \limsup_{k \rightarrow + \infty} \limsup_{n \rightarrow + \infty}  \frac{1}{\rho_k} \int_{(Q'_{\nu_0}(0, \rho_k))_{,1}} f \left (\frac{x_\alpha}{h_n},x_3, \nabla_{h_n} v_n \right ) \, dx.
\]
It is worth mentioning that the construction process to get $v_n$ starting from $u_n$ does not affect the manifold constraint.
\\
By means of a change of variable, setting $w_{n,k} = v_n(\rho_k x_\alpha, x_3)$ we get
\[
\frac{d \mu}{d \mathcal{H}^{1} \mres S_u}(x_0) \ge \, \limsup_{k \rightarrow + \infty} \limsup_{n \rightarrow + \infty}  \rho_k \int_{Q'_{\nu_0,1}} f \left (\frac{\rho_k \, x_\alpha}{h_n},x_3, \frac{1}{\rho_k} \nabla_{\alpha} w_{n,k}\Bigg| \frac{1}{h_n}\nabla_{3} w_{n,k} \right ) \, dx.
\]
At this point, let us define
\[
u_0(x) := 
\left \{
\begin{array}{lll}
\!\!\!\!\!\! & s_0^+ \qquad & \textnormal{if $x \cdot \nu_0 > 0$}\\
\!\!\!\!\!\! & s_0^- \qquad & \textnormal{if $x \cdot \nu_0 \le 0$}
\end{array}
\right.
\]
By using \eqref{BM(6.25)}, we obtain
\[
\lim_{k \rightarrow + \infty} \lim_{n \rightarrow + \infty} \int_{Q'_{\nu_0,1}} |w_{n,k} - u_0| \, dx = 0.
\]
Exploiting a standard diagonal argument, we are able to find a sequence $n_k \rightarrow + \infty$ such that $(h_k)_k:=(h_{n_k})_k$, $\delta_k := h_{n_k}/\rho_k \rightarrow 0,$ $w_k := w_{n_k,k} \in \mathcal{D}(Q_{\nu_0,1}; \mathbb{R}^3)$ converges to $u_0 \in L^1(Q_{\nu_0,1}; \mathbb{R}^3)$ and
\begin{align}
    \label{BM(6.27)}
    \frac{d \mu}{d \mathcal{H}^{1} \mres S_u}(x_0) &\ge \, \limsup_{k \rightarrow + \infty}  \rho_k \int_{Q'_{\nu_0,1}} f \left (\frac{x_\alpha}{\delta_k},x_3, \frac{1}{\rho_k}\nabla_{\alpha} w_{k}\Bigg| \frac{1}{h_k}\nabla_{3} w_{k} \right ) \, dx\nonumber\\
    &=\limsup_{k \rightarrow + \infty}  \rho_k \int_{Q'_{\nu_0,1}} f \left (\frac{x_\alpha}{\delta_k},x_3, \frac{1}{\rho_k}\nabla_{\delta_k} w_{k}\right ) \, dx
\end{align}
Exploiting (H4), the positive 1-homogeneity of $f^{\infty}(y, \cdot)$ and H\"older's inequality (being $0 < q < 1$)
\begin{eqnarray}
&& \int_{Q'_{\nu_0,1}} \left | \rho_k f \left ( \frac{x_\alpha}{\delta_k},x_3, \frac{1}{\rho_k}\nabla_{\delta_k} w_{k}\right ) - f^{\infty} \left ( \frac{x_\alpha}{\delta_k},x_3, \nabla_{\delta_k} w_{k}\right )\right | \, dx \nonumber \\
&=& \int_{Q'_{\nu_0,1}} \left | \rho_k f \left ( \frac{x_\alpha}{\delta_k},x_3,\frac{1}{\rho_k}\nabla_{\delta_k} w_{k}\right ) - \rho_k f^{\infty} \left ( \frac{x_\alpha}{\delta_k},x_3, \frac{1}{\rho_k}\nabla_{\delta_k} w_{k}\right )\right | \, dx\nonumber\\
&\le & \, C \, \rho_k \int_{Q'_{\nu_0,1}} (1 + \rho_k^{q-1} |\nabla_{\delta_k} w_k|^{1 - q} ) \, dx \nonumber\\
& \le & C \left ( \rho_k + \rho_k^q \|\nabla_{\delta_k} w_k\|^{1 - q}_{L^1(Q'_{\nu_0,1}; \mathbb{R}^{3 \times 3})}\right ).\label{BM(6.28)}
\end{eqnarray}
On the other hand, from \eqref{BM(6.27)} and the coercivity condition (H2), it follows that the sequence $(\nabla_{\delta_k} w_k)_k$ is uniformly bounded in $L^1(Q'_{\nu_0,1}; \mathbb{R}^{3 \times 2}).$ Putting together \eqref{BM(6.27)} and \eqref{BM(6.28)}, we obtain
\begin{equation}
\label{BM(6.29)}
\frac{d \mu}{d \mathcal{H}^{1} \mres S_u}(x_0) \ge \, \limsup_{k \rightarrow + \infty}  \int_{Q'_{\nu_0,1}} f^{\infty} \left (\frac{x_\alpha}{\delta_k},x_3, \nabla_{\delta_k} w_{k} \right ) \, dx.
\end{equation}
{\sc step 2.} In this second and last step it remains to modify the value of $w_k$ on a neighbourhood of $\partial Q'_{\nu_0,}$ with the aim to get an admissible test function for the surface energy density. The computations follow similarly as in \cite[Lemma 5.2]{AEL07}.
\\
We consider $\gamma \in \mathcal{G}(s_0^+, s_0^-)$ (see \eqref{geodesic} and \eqref{x_nu}) and set
\[
\psi_k(x) := \gamma \left ( \frac{x_{\nu_0}}{\delta_k}\right ).
\]
Using a De Giorgi type slicing argument, we shall modify $w_k$ in order to get a function which matches $\psi_k$ to $\partial Q'_{\nu_0,1}.$ In order to reach this purpose, we introduce
\[
\begin{array}{ll}
& r_k := \|w_k - \psi_k\|^{1/2}_{L^1(Q'_{\nu_0,1}; \mathbb{R}^3)}\\[3mm]
& \displaystyle M_k := k \, \left [1 + \|w_k\|_{W^{1,1}(Q'_{\nu_0,1}; \mathbb{R}^{3})} + \|\psi_k\|_{W^{1,1}(Q'_{\nu_0,1}; \mathbb{R}^{3})} \right ]\\[3mm]
& \displaystyle \ell_k := \frac{r_k}{M_k}
\end{array}
\]
As long as $\psi_k$ and $w_k$ converge to $u_0$ in $L^1(Q'_{\nu_0,1}; \mathbb{R}^{d}),$ we have that $r_k \rightarrow 0$ and it is possible to assume that $0 < r_k < 1.$ We set
\[
Q_k^{(i)} := ((1 - r_k + i \ell_k) Q'_{\nu_0})_{,1} \qquad \textnormal{for $i = 0, \dots, M_k$.}
\]
For every $i \in \{1, \dots, M_k\},$ we consider a cut-off functions $\varphi_k^{(i)} \in \mathcal{C}^{\infty}_c((1 - r_k + i \ell_k) Q'_{\nu_0}; [0,1])$ and we extend $\varphi_k^{(i)}$ to $Q_k^{(i)}$ setting $\varphi_k^{(i)}$ constant in the third variable. For simplicity of notation we still denote $\varphi_k^{(i)}$ this extension. By construction $\varphi_k^{(i)} = 1$ on $Q_k^{(i-1)}$ and $|\nabla_{\delta_k} \varphi_k^{(i)}| \le \, c/\ell_k.$ We further define
\[
z_k^{(i)} := \varphi_k^{(i)} w_k + (1 - \varphi_k^{(i)}) \psi_k \in W^{1,1}(Q_{\nu_0,1}; \mathbb{R}^3),
\]
in such a way that
\[
z_k^{(i)} = w_k \qquad \textnormal{in $Q_{k}^{(i-1)} \,\,\,$ and } \qquad z_k^{(i)} = \psi_k \qquad \textnormal{in $Q'_{\nu_0,1} \setminus Q_{k}^{(i)}.$}
\]
Since $z_k^{(i)}$ is smooth outside a finite union of sets contained in some $1-$dimensional submanifolds and $z_k^{(i)}(x) \in co(\mathcal{M})$ for a.e. $x \in Q'_{\nu_0,1},$ it is possible to apply a similar argument to \cite[Proposition 2.1]{BMbv} to obtain new functions $\hat{z}_k^{(i)} \in W^{1,1}(Q'_{\nu_0,1}; \mathcal{M})$ such that
\[
\hat{z}_k^{(i)} = z_k^{(i)} \,\,\, \textnormal{on} \,\,\, (Q'_{\nu_0,1} \setminus Q_k^{(i)}) \cup Q_k^{(i-1)}
\]
and
\begin{eqnarray*}
\int_{Q_k^{(i)} \setminus Q_k^{(i-1)}} |\nabla_{\delta_k} \hat{z}_k^{(i)}| \, dx &\le &\, C_* \, \int_{Q_k^{(i)} \setminus Q_k^{(i-1)}} |\nabla_{\delta_k} z_k^{(i)}| \,dx \\
&\le & \, C_* \int_{Q_k^{(i)} \setminus Q_k^{(i-1)}} \left (|\nabla_{\delta_k} w_k| + |\nabla_{\delta_k} \psi_k| + \frac{1}{\ell_k} |w_k - \psi_k| \right ) \, dx.
\end{eqnarray*}
In particular $\hat{z}_k^{(i)} \in \mathcal{B}_{\delta_k}(s_0^+, s_0^-, \nu_0)$ and by the growth condition (H2)
\begin{eqnarray*}
&& \int_{Q'_{\nu_0,1}} f^{\infty} \left (\frac{x_\alpha}{\delta_k},x_3, \nabla_{\delta_k} \hat{z}_k^{(i)} \right ) \, dx \le  \, \int_{Q'_{\nu_0,1}} f^{\infty} \left (\frac{x_\alpha}{\delta_k},x_3, \nabla_{\delta_k} w_k \right ) \, dx\\
&&+ C \, \int_{Q'_{\nu_0,1} \setminus Q_k^{(i)}} |\nabla_{\delta_k} \psi_k| \, dx + C \int_{Q_k^{(i)} \setminus Q_k^{(i-1)}}  \left (|\nabla_{\delta_k} w_k| + |\nabla_{\delta_k} \psi_k| + \frac{1}{\ell_k} |w_k - \psi_k|  \right ) \, dx.
\end{eqnarray*}
Summing up over all $i = 1, \dots, M_k$ and dividing by $M_k,$ we get that
\begin{eqnarray*}
&& \frac{1}{M_k} \sum_{i=1}^{M_k} \int_{Q'_{\nu_0,1}} f^{\infty} \left (\frac{x_\alpha}{\delta_k},x_3, \nabla_{\delta_k} \hat{z}_k^{(i)} \right ) \, dx \le  \, \int_{Q'_{\nu_0,1}} f^{\infty} \left (\frac{x_\alpha}{\delta_k},x_3, \nabla_{\delta_k} w_k \right ) \, dx\\
&&+ C \, \int_{Q'_{\nu_0,1} \setminus Q_k^{(0)}} |\nabla_{\delta_k} \psi_k| \, dx + \frac{C}{k} + C \, \|w_k - \psi_k\|^{1/2}_{L^1(Q'_{\nu_0,1}; \mathbb{R}^d)}.
\end{eqnarray*}
Since
\[
\int_{Q'_{\nu_0,1} \setminus Q_k^{(0)}} |\nabla \psi_k| \, dx \le \, {\bf d}_{\mathcal{M}} (s_0^+, s_0^-) \mathcal{H}^{2} \left ( (Q'_{\nu_0,1} \setminus Q_k^{(0)}) \cap \{x \cdot \nu = 0\}\right ) \rightarrow 0
\]
as $k \rightarrow + \infty,$ there exists a sequence $\eta_k \rightarrow 0^+$ such that
\[
\frac{1}{M_k} \sum_{i=1}^{M_k} \int_{Q'_{\nu_0,1}} f^{\infty} \left (\frac{x_\alpha}{\delta_k},x_3, \nabla_{\delta_k} \hat{z}_k^{(i)} \right ) \, dx \le  \, \int_{Q'_{\nu_0,1}} f^{\infty} \left (\frac{x_\alpha}{\delta_k},x_3, \nabla_{\delta_k} w_k \right ) \, dx + \eta_k.
\]
Hence, for each $k \in \mathbb{N}$ it is possible to find some index $i_k \in \{1, \dots M_k\}$ satisfying 
\begin{equation}
\label{BM(6.30)}
\int_{Q'_{\nu_0,1}} f^{\infty} \left (\frac{x_\alpha}{\delta_k},x_3, \nabla_{\delta_k} \hat{z}_k^{(i_k)} \right ) \, dx \le  \, \int_{Q'_{\nu_0,1}} f^{\infty} \left (\frac{x_\alpha}{\delta_k},x_3, \nabla_{\delta_k} w_k \right ) \, dx + \eta_k.
\end{equation}
Putting together \eqref{BM(6.29)} and \eqref{BM(6.30)}, we finally obtain
\[
\frac{d \mu}{d \mathcal{H}^{1} \mres S_u}(x_0) \ge \, \limsup_{k \rightarrow + \infty}  \int_{Q_{\nu_0}} f^{\infty} \left (\frac{x_\alpha}{\delta_k},x_3, \nabla_{\delta_k} \hat{z}_{k}^{(i_k)} \right ) \, dx.
\]
Since $\hat{z}_k^{(i_k)} \in \mathcal{B}(s_0^+, s_0^-, \nu_0)$ (see \eqref{B set} for definition), we infer by Proposition \ref{BMBV3.2_3.3}  that
\[
\frac{d \mu}{d \mathcal{H}^{1} \mres S_u}(x_0) \ge \, \theta(s_0^+, s_0^-, \nu_0)
\]
which finally concludes the proof.
\end{proof}

\noindent
From the proofs of Theorem \ref{BVcase}, it should be clear that both results hold in dimensions $N \ge 3$. In particular, we can assume $N,k,d\in \N$ with the $1\le k <N$, $\omega\subset \R^{N-k}$, $\Omega:= \omega \times (-\frac{1}{2}, \frac{1}{2})^k$, and $\mathcal{M}$ submanifold of $\R^d$.

\color{black}

    \section*{Acknowledgment}
The authors gratefully acknowledge support from INdAM GNAMPA.
	A.T. and L. L. \,have been partially supported through the INdAM-GNAMPA 2025 project ``Minimal surfaces: the Plateau problem and behind'' cup E5324001950001 and the project: Geometric-Analytic Methods for PDEs and Applications (GAMPA) , ref. 2022SLTHCE – cup E53D2300588 0006 - funded by European Union - Next Generation EU within the PRIN 2022 program (D.D. 104 - 02/02/2022 Ministero dell’Università e della Ricerca). This manuscript reflects only the authors’ views and opinions and the Ministry cannot be considered responsible for them. 
    E. Z. has been supported by PRIN 2022: Mathematical Modelling of Heterogeneous Systems (MMHS)
- Next Generation EU CUP B53D23009360006 and by INdAM GNAMPA Project 2024 'Composite materials and microstructures'.

\end{document}